\documentclass[12pt]{amsart}
\usepackage{amssymb, amsmath, mathtools}
\usepackage{stmaryrd}

\usepackage[dvipsnames]{xcolor}
\usepackage{color}
\usepackage[utf8]{inputenc}
\usepackage[T1]{fontenc}
\usepackage{fullpage}
\usepackage[normalem]{ulem}
\usepackage{amsmath,amscd,mathrsfs,enumerate,comment}
\usepackage{euler, amsfonts, amssymb, latexsym, epsfig,epic}
\usepackage[colorlinks=true, citecolor=blue, linkcolor=red,pagebackref, hyperindex]{hyperref}
\usepackage[all,cmtip]{xy}

\newtheorem{theorem}{Theorem}[section]
\newtheorem{theoremx}{Theorem}
 
\newtheorem{lemma}[theorem]{Lemma}
\newtheorem{corollary}[theorem]{Corollary}
\newtheorem{proposition}[theorem]{Proposition}
\newtheorem{claim}[theorem]{Claim}
\theoremstyle{definition}
\newtheorem{example}[theorem]{Example}

\newtheorem{definition}[theorem]{Definition}
\newtheorem{question}[theorem]{Question}

\newtheorem{notation}[theorem]{Notation}
\theoremstyle{remark}
\newtheorem{remark}[theorem]{Remark}

\newcommand{\kk}{\mathbb{K}}

\DeclareMathOperator{\depth}{depth}
\DeclareMathOperator{\edim}{edim}
\DeclareMathOperator{\height}{ht}
\DeclareMathOperator{\bigheight}{bight}
\DeclareMathOperator{\Max}{Max}

\DeclareMathOperator{\Proj}{Proj}
\DeclareMathOperator{\Supp}{Supp}
\DeclareMathOperator{\Ann}{Ann}

\DeclareMathOperator{\ord}{ord}

\DeclareMathOperator{\Hom}{Hom}
\DeclareMathOperator{\Spec}{Spec}
\DeclareMathOperator{\MaxSpec}{MaxSpec}
\DeclareMathOperator{\Min}{Min}

\DeclareMathOperator{\Ext}{Ext}

\DeclareMathOperator{\loewy}{\ell\ell}

\DeclareMathOperator{\coker}{coker}

\DeclareMathOperator{\e}{e}
\DeclareMathOperator{\fpt}{fpt}
\DeclareMathOperator{\lct}{lct}
\DeclareMathOperator{\dfpt}{dfpt}
\DeclareMathOperator{\mfpt}{\mu fpt}
\DeclareMathOperator{\Ft}{c}

\DeclareMathOperator{\s}{s}

\DeclareMathOperator{\sdim}{sdim}

\DeclareMathOperator{\IN}{in}

\DeclareMathOperator{\gr}{gr}

\def\p{\mathfrak{p}}
\def\q{\mathfrak{q}}
\def\m{\mathfrak{m}}
\def\n{\mathfrak{n}}
\def\a{\mathfrak{a}}

\def\FF{\mathbb{F}}
\def\R{\mathbb{R}}
\def\P{\mathbb{P}}
\def\QQ{\mathbb{Q}}
\def\ZZ{\mathbb{Z}}
\def\NN{\mathbb{N}}

\def\F{\mathbb{F}}

\def\OO{\mathcal{O}}

\def\cB{\mathcal{B}}
\def\cR{\mathcal{R}}
\def\cS{\mathcal{S}}
\def\cF{\mathcal{F}}
\def\cD{\mathcal{D}}
\def\CC{\mathbb{C}}

\newcommand{\ecodim}{\operatorname{ecodim}}

\newcommand{\cotimes}[1]{\mathbin{\widehat{\otimes_{#1}}}}

\newcommand{\FDer}[1]{\stackrel{#1}{\to}}
\renewcommand{\geq}{\geqslant}
\renewcommand{\leq}{\leqslant}
\newcommand{\PP}{\mathbb{P}}

\newcommand{\ps}[1]{\llbracket {#1} \rrbracket}

\newcommand{\soc}{\operatorname{soc}}

\newcommand{\frpe}[1]{{#1}^{[p^e]}}

\newcommand{\ul}{\underline}

\DeclareMathOperator{\cP}{\mathcal{P}}

\begin{document}

\title{The defect of the F-pure threshold}

\author[A. De Stefani]{Alessandro De Stefani}
\address{Dipartimento di Matematica, Universit{\`a} di Genova, Via Dodecaneso 35, 16146 Genova, Italy}
\email{alessandro.destefani@unige.it}

\author[L. N\'u\~nez-Betancourt]{Luis N\'u\~nez-Betancourt}
\address{Centro de Investigaci\'on en Matem\'aticas, Guanajuato, Gto., M\'exico}
\email{luisnub@cimat.mx}

\author[I. Smirnov]{Ilya Smirnov}
\address{BCAM -- Basque Center for Applied Mathematics, Mazarredo 14, 48009 Bilbao, Basque Country -- Spain}
\address{Ikerbasque, Basque Foundation for Science, Plaza Euskadi 5, 48009 Bilbao, Basque Country -- Spain}
\email{ismirnov@bcamath.org}

\begin{abstract}
Introduced by Takagi and Watanabe, F-pure thresholds are invariants defined in terms of the Frobenius homomorphism. While they find applications in various settings, they are primarily used as a \emph{local} invariant. The purpose of this note is to start filling this gap by 
opening the study of its behavior on a scheme. To this end, we define the \emph{defect of the F-pure threshold} of a local ring $(R,\m)$ by setting $\dfpt(R)=\dim (R) - \fpt (\m)$. It turns out that this invariant defines an upper semi-continuous function on a scheme and satisfies Bertini-type theorems.
We also study the behavior of the defect of the F-pure threshold under flat extensions and after blowing up the maximal ideal of a local ring.
\end{abstract}

\maketitle

\hypersetup{linkcolor=black}
\setcounter{tocdepth}{1}
\tableofcontents

\section{Introduction}

The origin of F-pure thresholds is in the theory of singularities of pairs. 
The notion of \emph{F-purity}, introduced by Hochster and Roberts \cite{HRFpurity}, was extended to pairs by Takagi \cite{TakagiAdj}.
Takagi and Watanabe \cite{TakagiWatanabe}  defined the F-pure threshold of an ideal $I$ of an F-finite ring $R$ to be $\fpt (I) = \sup \{t \mid (R, I^t) \text{ is F-pure}\}$. 
Remarkably, one obtains the same invariant by using other definitions of 
F-purity for pairs, for example, Schwede's notion of \emph{sharp F-purity} \cite{Schwede} is often easier to work with.

An appealing feature of the theory is its inherent blend of algebra and geometry. Further importance comes from its close connections with the theory of log-canonical thresholds, an important invariant of birational geometry notable for its connections with the minimal model program  \cite{Shokurov, Birkar} and its role in the theory of the normalized volumes  \cite{BlumLiu, Liu18}.
It is conjectured that the class of F-pure singularities corresponds to the class of log-canonical singularities via reduction mod $p$ \cite{HaraWatanabe, HaraYoshida, FujinoTakagi}, and that the thresholds quantify the correspondence: the F-pure thresholds of the reductions mod $p$ of a singularity approximate the log-canonical threshold, i.e.,
$\lct (R; I)
= \lim_{p \to \infty} \fpt (R/p, IR/p)$ 
\cite{TakagiWatanabe}. Yuchen Liu has been working on a positive characteristic analogue of the normalized volume that utilizes F-pure thresholds instead of log-canonical thresholds \cite{Liu}.

F-pure thresholds can be used to study the singularities of a local ring $(R, \m)$ in two  different ways: first, one can use $\fpt (I;S)$ if $R = S/I$ is a quotient of a regular local ring $S$. Alternatively, one can use $\fpt (R) \coloneqq \fpt (\m)$ as an intrinsic measure of singularity. 
The first approach has enjoyed a lot of attention \cite{HaraYoshida, BMS-MMJ, BMS-Hyp} and is especially popular in the case of hypersurfaces because this is where the theory of log canonical thresholds has its origin under the name the \emph{complex singularity exponent}. In this paper, however, we focus on the second point of view. 
The reason is that the former approach does not seem to detect whether $R$ is regular, while it was shown already in the foundational paper of Takagi and Watanabe \cite{TakagiWatanabe} that $\fpt (R)$ does measure singularities:  in general, $\fpt (R) \leq \dim (R)$ and equality is only possible in a regular local ring.  Moreover, as $\fpt (R)$ approaches $\dim (R)$, the singularity gets milder (for example, Proposition~\ref{prop: TW reduction} and Proposition~\ref{prop: condition F-regular}). Thus $\fpt(R)$ is fitting naturally in the framework of \emph{measures of singularities} that includes various multiplicities or, for example, the F-signature \cite{SmithVDB,HLMCM}. 

Our article provides a further development within the framework by starting to study the F-pure threshold globally, as a function $\Spec (R) \ni \q \mapsto \fpt (R_\q)$. A natural direction is to study the stratification of a scheme by the values of an invariant. Unfortunately, the behavior of  the F-pure threshold after localization is quite erratic, as it can either increase or decrease. For instance, take $R=\kk\ps{x,y,z}/(x^3+y^3+z^3)$, where $\kk$ is an F-finite field of characteristic $p \equiv 1 \bmod 3$, and let $\m = (x,y,z)$. Then for any $\p \in \Spec(R)$ with $(0) \subsetneq \p \subsetneq \m$ one has
\[
\fpt(R_\m) = 0, \ \fpt(R_{\p}) = 1, \text{ and } \fpt(R_{(0)}) = 0.
\]
We overcome this issue by blending the F-pure threshold into a more well-behaved invariant. We define the \emph{defect of the F-pure threshold} as
$$
\dfpt(R) \coloneqq \dim(R)-\fpt(R).
$$
At a single point the defect of the F-pure threshold carries the same information as the F-pure threshold. However, we show that it satisfies many good properties as a global invariant on $\Spec(R)$. Many of the existing results are stated more naturally for this invariant rather than for the F-pure threshold itself. The following theorems summarize some of our findings. 

\begin{theoremx}\label{MainThmGlobal}
Let $X$ be a locally F-finite and F-pure scheme of characteristic $p > 0$. Then
\begin{enumerate}
\item $x \mapsto \dfpt (\OO_{X, x})$ defines an upper semi-continuous function (Theorem \ref{thm semicont}); 
\item in addition, if $X$ is $\QQ$-Gorenstein and Cohen-Macaulay, then the above function defines a locally finite constructible stratification (Theorem \ref{thm constructible});
\item if $X$ is an irreducible affine scheme, then there is a 
well-defined global invariant $\dfpt(X)$ which agrees with $\max\{\dfpt(\OO_{X,x}) \mid x \in X\}$ (Theorem \ref{thm: global (new)}); 
\item if $X \subseteq \PP^n_{\kk}$ is an irreducible quasi-projective variety, with $\kk$ an algebraically closed field, and $\lambda \in \R$ is such that $\dfpt (\OO_{X, x}) < \lambda$ for all $x \in X$, then for a general hyperplane $H \subseteq \P^n_{\kk}$
we still have $\dfpt (\OO_{X \cap H, x}) < \lambda$ for all $x \in X \cap H$ (Corollary\ref{coroll Bertini}). 
Moreover, $\max_{x \in X \cap H} \{\dfpt (\OO_{X \cap H, x})\} \leq \max_{x \in X} \{\dfpt (\OO_{X,x})\}$
if $X$ is Gorenstein and normal (Corollary~\ref{coroll Bertini Gor}) or $H$ is very general. 
\end{enumerate}
\end{theoremx}

In birational geometry it is important to understand the behavior of a measure of singularities under blow-ups. 
Little is known about the behavior of  F-invariants in this context
and there are simple examples showing that they behave pathologically after blowing up the closed point of a local ring \cite{MPST}. In contrast, we are able to establish a degree of control for the defect of F-pure threshold.

\begin{theoremx} \label{MainThmBlowup}
Let $X$ be a locally F-finite and F-pure scheme of characteristic $p > 0$ and 
suppose that $\pi \colon Y \to X$ is the blow-up of a closed point $x \in X$.
Assume that the exceptional divisor $E$ of $\pi$ is arithmetically Gorenstein and globally F-split\footnote{We note that that this property is equivalent to the associated graded ring of the maximal ideal being Gorenstein and F-pure}. Then
$\OO_{X, x}$ is F-pure and 
$\dfpt(Y) \leq \dfpt(\OO_{X,x}) = \dfpt (E)$ (Theorem \ref{thm associated graded}, Proposition \ref{prop: blowup}). 

Moreover, a similar result holds for the F-signature:
if $E$ is arithmetically Gorenstein and globally F-regular, then $\OO_{X,x}$ is strongly F-regular and 
 $\s(E) \leq \s(\OO_{X,x})$ (Theorem \ref{thm associated graded}).
\end{theoremx}

We also obtain new results about the local behavior of $\dfpt(R)$.

\begin{theoremx}\label{MainThmLocal}
Let $(R, \m)$ be an F-finite F-pure local ring of characteristic $p > 0$. Then
\begin{enumerate}
\item if $\dfpt (R) < 2$ and $R$ is normal then $R$ is strongly F-regular (Proposition~\ref{prop: condition F-regular});
\item if $R$ is Gorenstein and $f$ is a parameter such that $R/(f)$ is F-pure, then $\dfpt (R) \leq \dfpt (R/(f))$ (Proposition~\ref{prop hyperplane}); 
\item if $(A, \m_A) \to (R, \m_R)$ is a flat local map of F-finite F-pure rings with $R/\m_A R$ reduced, then $\dfpt (A) \leq \dfpt (R)$ (Proposition~\ref{prop: ineq flat 1});
\item if $(A,\m_A) \to (R,\m_R)$ is a flat local map of F-finite F-pure rings with $R/\m_A R$ Gorenstein, then $\dfpt (R) \leq \dfpt (R) + \dfpt(R/\m_A R)$ (Theorem \ref{gorenstein flat});
\item in addition, if $R/\m_A R$ is regular, then $\dfpt(A) = \dfpt(R)$ (Corollary \ref{coroll A1}).
\end{enumerate}
\end{theoremx}

These results are quantitative strengthenings of the corresponding results on the deformation \cite[Theorem~3.4]{FedderFputityFsing}, 
flat descent \cite[Proposition~5.13]{HRFpurity}, 
and flat ascent of F-purity \cite[Proposition~5.14]{MaLC}, \cite[Proposition~2.4]{Hashimoto}, \cite[Proposition~4.8]{SchwedeZhang}.

\subsection{Overview of the paper}
We now provide an outline of the paper in order to highlight some of the technical tools we develop. We start Section~\ref{sect: prelim} by giving preliminaries on F-singularities and F-pure thresholds which we then restate in terms of our invariant. Section~\ref{sect: differential} develops the key technical tool -- a formula for the defect of the F-pure threshold of $R$ using its presentation as a quotient of  a regular ring (Section~\ref{sect: main formula}). 
Specifically, if $(S, \m)$ is an F-finite regular local ring, and $R = S/I$ is F-pure, then 
we show in Theorem~\ref{thm: main formula} that the defect of the F-pure threshold can be computed using the sequence 
$e \mapsto \Theta_e \coloneqq \max \{t \mid I^{[p^e]}:_S I \subseteq \m^t + \m^{[p^e]}\}$.  A similar type of formula is well-known for the F-pure threshold, and can be seen as a refinement of Fedder's criterion for F-purity \cite{FedderFputityFsing}. Note that, when localizing, the left-hand side $I^{[p^e]}:_S I$ of our formula localizes as well, but the right-hand side does not. 

In order to control the right-hand side we employ the second key tool:
a reformulation of $\Theta_e$ in terms of differential operators (Definition~\ref{def: theta}). Namely, in (\ref{ssect: differential}) we show that there are submodules $ D^{(n-1,p^e)}_S$ of $D_S$, the ring of $\ZZ$-linear differential operators on a ring $S$, such that for any prime ideal $\p$ of $S$ the $\p$-primary component of $\p^n + \p^{[p^e]}$ coincides with
$\{s \in S \mid \delta(s) \in \p \text{ for all } \delta \in D^{(n-1,p^e)}_S\}$.
From this we deduce that an individual $\Theta_e$ has good global properties 
which we then extend to the global properties of the defect of the F-pure threshold by proving a uniform convergence result. 

The bulk of Section~\ref{sect: global} is devoted to proving Theorem~\ref{MainThmGlobal}: we start by proving semi-continuity utilizing uniform convergence results of Section~\ref{sect: differential}, and then proceed to develop the global version of the defect of the F-pure threshold. It should be noted, at this point, that for a local ring $(R, \m)$ 
the invariant $\mfpt (R) \coloneqq \edim (R) - \fpt (R)$, where $\edim(R) = \dim_{A/\m}(\m/\m^2)$ is the embedding dimension of $R$, also detects singularity, and is upper semi-continuous. However, its global theory is more problematic -- for instance, it lacks a global definition -- and this led us to choose $\dfpt (R)$ as the main invariant. 

A crucial technical result of Section~\ref{sect: global} is a global version of Fedder's criterion of F-purity stated using the machinery of differential operators. 

\begin{theoremx}\label{ThmFedderC}
Let $S$ be an F-finite regular ring, $I\subseteq S$ be an ideal, and $R=S/I$.
Then 
\begin{enumerate}
\item $R$ is F-pure if and only if $\left(D^{(\dim(S)(p^e-1)S, p^e)}_S \left( I^{[p^e]}:I\right)\right) = S$ for some  integer $e>0$ (Proposition \ref{prop: global Fedder});
\item and, if  $S=\kk[x_1,\ldots,x_n]$ with $\kk$  an F-finite field,
then $R$ is geometrically F-pure if and only if $\left(D^{(n(p^e-1), p^e)}_{S|\kk}\left( I^{[p^e]}:I \right)\right)= S$ for some integer $e>0$  (Proposition \ref{prop: geo Fedder}).
\end{enumerate}
\end{theoremx}

Theorem~\ref{ThmFedderC} should be compared with Fedder's work \cite[Theorem 1.13]{FedderFputityFsing}, where he used differential operators to prove that the F-pure locus of a finitely generated $\kk$-algebra is open in the max-spectrum. Theorem~\ref{ThmFedderC} works for every F-finite ring, and can be used to compute the F-pure locus in $\Spec(R)$. We establish a Bertini-type theorem for the defect of the F-pure threshold using the axiomatic framework of Cumino, Greco, and Manaresi \cite{CGM}. 

In Section~\ref{sect: local} we complete the proofs of the remaining parts of Theorem~\ref{MainThmLocal}, employ our characterization of the defect of F-pure threshold using the differential operators to study blow-up algebras and to produce Theorem~\ref{MainThmBlowup}, and finish the section by deriving several novel results for log-canonical thresholds in equal characteristic zero via reduction to prime characteristic techniques (\ref{ssect: modp}). 

The paper ends with various remarks in Section~\ref{sect: last}. 
Among them, we provide an example showing that one cannot obtain by the same recipe a semi-continuous invariant using the diagonal F-threshold $c^\m(\m)$, and an 
example showing that the $\dfpt$ and $\mfpt$ may attain maxima at different points of a scheme.

\subsection*{Acknowledgments}
We are indebted to Karen E. Smith for suggesting Corollary~\ref{cor: thanks Karen} and discussions in \ref{sub: thanks Karen}, to Shunsuke Takagi for suggestions on reductions mod $p > 0$,  and to Kevin Tucker for bringing Liu's work \cite{Liu} to our attention. We thank Florian Enescu and Linquan Ma for comments. In particular, we thank Ma for pointing out an inaccuracy in the published version of Lemma \ref{LemmaACC}. We also thank the anonymous referee for providing useful suggestions.

The first and second authors thank INdAM-GNSAGA for partially supporting a visit of the second author to the University of Genova. The first author was partially supported by the MIUR Excellence Department Project CUP D33C23001110001, and by INdAM-GNSAGA. The second author thanks SECIHTI, Mexico, for its support with Grants CBF 2023-2024-224 and CF-2023-G-33. The second author also thanks the  Centre de Recerca Matemàtica for the Lluis Santaló Fellowship that allowed him to visit Spain in 2022 and 2023.

The last author was supported by the State Research Agency of Spain  through grants PID2021-125052GA-I00 funded by MICIU/AEI/10.13039/501100011033 and FEDER, UE, RYC2020-028976-I funded by MICIU/AEI/10.13039/501100011033 and EI ESF "ESF Investing in your future”,  and EUR2023-143443 funded by MCIN/AEI/10.13039/501100011033 and the European Union NextGenerationEU/PRTR 
and also by the Basque Government through the BERC 2022-2025 program
and by the Ministry of Science and Innovation: BCAM Severo Ochoa accreditation
CEX2021-001142-S/MICIN/AEI/10.13039/501100011033.

\section{The defect of the F-pure threshold}\label{sect: prelim}
Throughout this manuscript, unless otherwise stated, all rings are commutative Noetherian rings of prime characteristic $p$. We use $\dim (R)$ to denote the Krull dimension of a ring $R$ and, if $(R,\m)$ is local, we use $\edim (R) = \dim_{A/\m}(\m/\m^2)$ to denote its embedding dimension. 

For an integer $e>0$, we let $F^e\colon R \to R$ denote the $e$-th iteration of the Frobenius map, that is, the ring endomorphism defined by $f \mapsto f^{p^e}$. We denote by $F^e_*(R)$ the module structure on $R$ induced by restriction of scalars via $F^e$. Given $f \in R$, we denote by $F^e_*(f)$ an element $f$ viewed through the Frobenius as en element of $F^e_*(R)$. Given $f,g \in R$ we then have $F^e_*(f)+F^e_*(g) = F^e_*(f+g)$ and $f \cdot F^e_*(g) = F^e_*(f^{p^e}g)$. 

When $R$ is reduced, we can identify $F^e_*(R)$ with $R^{1/p^e}$, the ring of $p^e$th roots of elements of $R$ inside an algebraic closure of the total ring of fractions of $R$. In this way, $F^e$ can be identified with the natural inclusion $R \hookrightarrow R^{1/p^e}$, and the $R$-module structure of $F^e_*(R)$ with that on $R^{1/p^e}$ induced by such an inclusion.

\subsection{Preliminaries on F-finite rings}

We recall that a ring of positive characteristic is called \emph{F-finite} if the Frobenius map is a finite morphism. Equivalently, $F^e_*(R)$ (or $R^{1/p^e}$ in the reduced case) is a finitely $R$-module for some (equivalently, for all) $e>0$. It was shown by Kunz \cite{F-finExc} that an F-finite ring is excellent and Gabber \cite{Gabber} showed that an F-finite ring is a quotient of a regular F-finite ring. Proposition \ref{prop: equidimensional Gabber} refines this slightly for a latter use. First, we recall the following notions.
\begin{definition} A ring $R$ of finite Krull dimension is said to be
\begin{itemize}
\item \emph{biequidimensional} if all maximal chains of prime ideals have equal length,
\item \emph{equidimensional} if $\dim(R/\p) = \dim(R)$ for all $\p \in \Min (R)$,
\item \emph{coequidimensional} if $\dim(R_\m) = \dim(R)$ for all maximals ideal $\m$ of $R$.
\end{itemize}
\end{definition}
If $R$ is biequidimensional, then it is both equidimensional and coequidimensional.

\begin{proposition}\label{prop: equidimensional Gabber}
Let $R$ be a biequidimensional F-finite ring. Then there is a presentation $R = S/I$ where $S$ is an F-finite coequidimensional regular ring. 
\end{proposition}
\begin{proof}
We start with any presentation $R\cong S/I$ where $S$ is an F-finite regular ring. Any regular ring is a finite product $S_1 \times \ldots \times S_n$ of regular domains, all of them F-finite in our assumptions. Thus, we have that $I=J_1 \times \ldots \times J_n$ for some ideals $J_i \subseteq S_i$. Note that $S$ is coequidimensional if and only if each factor $S_i$ is coequidimensional of dimension $\dim(S) = \max\{\dim(S_i) \mid i=1,\ldots,n\}$. For each $i \in \{1,\ldots,n\}$ we let $d_i=\dim(S) -\dim(S_i)$, and we consider the F-finite regular domain $S_i'=S_i\ps{x_{i,1},\ldots,x_{i,d_i}}$. Moreover, we let $J_i' = J_iS_i' + (x_{i,1},\ldots,x_{i,d_i})$. Note that $S_i/J_i \cong S_i'/J_i'$, and therefore $S'/J' \cong R$ where $S'=S_1' \times \ldots S_n'$ and $J'=J_1' \times \ldots \times J_n'$. With this reduction, we have that $S'$ is coequidimensional if and only if $S_i'$ is coequidimensional for each $i\in\{1,\ldots,n\}$. This allows us to reduce to the case in which the F-finite regular ring is a domain. For simplicity, we  make this assumption directly on $S$. We let $h=\height(I)$, $H=\bigheight(I)$, and we write $I = \bigcap_{j=h}^H I_j$, where $I_j$ is the intersection of the components of $I$ of pure height $j$. For any $j \in \{h,\ldots,H\}$ let $L_j = \bigcap_{i \ne j} I_i$, and observe that we must have $I_j+L_j= S$. In fact, if not, there would be a maximal ideal $\m$ of $S$ containing the sum. Note that $\Min(I) = \Min(I_j) \cup \Min(L_j)$, otherwise $R_\m$ would not be equidimensional because $\dim((S/Q)_\m) \ne \dim((S/Q')_\m)$ for any $Q \in \Min(I_j)$ and $Q' \in \Min(L_j)$. 

It follows from the observation above that $S/I \cong \prod_{j=h}^H S/I_j$. We now claim that for each $j$ we can find a coequidimensional F-finite regular domain $T_j$ and an ideal $I_j' \subseteq T_j$ such that $T_j/I_j' \cong S/I_j$. We let $\mathcal M_j$ be the set of maximal ideals of $S$ which contain $I_j$, and we let $W_j = S \smallsetminus \bigcup_{\m \in \mathcal M_j} \m$. Our previous observation shows that $\mathcal M_j \cap \mathcal M_{j'} = \emptyset$ for $j \ne j'$. Observe that for any $\m \in \mathcal M_j$ we must have $\dim(S_\m) = \dim((S/I_j)_\m) + \height(I_j S_\m) = \dim((S/I)_\m) + \height(I_j) = \dim(R) + j$, where for the second equality we use that the sets $\mathcal M_j$ are pairwise disjoint to conclude that $(S/I_j)_\m \cong (S/I)_\m$, while for the third equality we use that $R$ is coequidimensional. This shows that there exists a constant $t_j \geq 0$ only depending on $j=h,\ldots,H$ such that $t_j = \dim(S) - \dim(S_\m)$ for all $\m \in W_j$. In particular, we have that $t_j = \dim(S) - \dim(S_{W_j})$. We finally let $T_j = S_{W_j}\ps{y_{j,1},\ldots,y_{j,t_j}}$, which is an F-finite regular domain such that $\dim((T_j)_\m) = \dim(S)$ for all $\m \in \Max(T_j)$. In particular, $S'=T_h \times \ldots \times T_H$ is an F-finite coequidimensional regular ring of dimension $\dim(S)$. We let $I_j' = I_jT_j+(y_{j,1},\ldots,y_{j,t_j})$, and observe that $T_j/I_j' \cong S_{W_j}/I_jS_{W_j} \cong S/I_j$. Therefore if we let $I'=I_h' \times \ldots \times I_H'$ we finally have that $R \cong S/I \cong S'/I'$, as desired.
\end{proof}

\subsection{F-pure and strongly F-regular singularities}

\begin{definition} \label{DefFPure} $R$ is said to be {\it F-pure} if the Frobenius map $F\colon R \to R$ is a pure ring map. That is, if the map $R \otimes_R M \to R \otimes_R F_*(R)$ is injective for every $R$-module $M$. The ring $R$ is called {\it F-split} if the map $R \to F_*(R)$ splits as a map of $R$-modules.
\end{definition} 

If $R$ is F-split, then it is F-pure. The converse is false in general, even for regular rings \cite{DattaMurayamaTate}, but it holds if $R$ is either F-finite \cite[Corollary 5.3]{HRFpurity}, or essentially of finite type over a complete local ring \cite[Theorem 3.1.1]{DattaMurayamaSolid}.

We recall the following classical result of Fedder.

\begin{theorem}[{\cite[Lemma 1.6 and Proposition 1.7]{FedderFputityFsing}}]\label{FedderCriterion}
Let $(S,\m)$ be an F-finite regular local ring, $I\subseteq S$ an ideal and $R=S/I$.
There is a bijective correspondence
$(I^{[p^e]}:_SI)/I^{[p^e]} \cong \Hom_R (F_*^e(R), R)$ given by 
$f  \mapsto \Phi(F_*^e(f \cdot -))$, where $\Hom_S (F_*^e(S), S) \cong F_*^e (S) \cdot \Phi$. As a consequence, $R$ is F-pure if and only if 
$I^{[p^e]}:_S I\not\subseteq \m^{[p^e]}.$
\end{theorem}

\begin{definition} An F-finite domain $R$ is said to be {\it strongly F-regular} if, for every $0\ne c \in R$, there exists $e>0$ and an $R$-linear map $\varphi\colon F^e_*(R) \to R$ such that $\varphi(F^e_*(c)) = 1$. 
\end{definition}

The property of being F-split can be detected and studied via the \emph{splitting ideals}.

\begin{definition}[{\cite{AE}, \cite{YaoFsig}}]
Let $(R,\m)$ be an F-finite local ring. We let
\[ I_e(R) \coloneqq \left\{ f\in R  \mid \psi(F^e_*(f))\in \m \text{ for all $R$-linear maps } \psi\colon F^e_*(R)\to R  \right\}. \] 
\end{definition}

\begin{remark} \label{rem: Fedder} If $(S,\m)$ is a regular local ring, and $R=S/I$, then it follows from the work of Fedder that $I_e(R) = \left(\m^{[p^e]}:_S(I^{[p^e]}:_SI)\right)/I$ for all $e \geq 0$ \cite{FedderFputityFsing}. 
\end{remark} 
By definition of $I_e(R)$, if $f \notin I_e(R)$ for some $e$, then there exists a map $\psi \colon F^e_*(R) \to R$ that splits the $R$-module inclusion $F^e_*(f) \cdot R \subseteq F^e_*(R)$. Hence, $R$ is F-pure if and only if $I_e (R) \neq R$ for all (some) $e >0$. It also follows from the definition that 
$I_{e}(R)^{[p]} \subseteq I_{e+1}(R)$ and that 
$\psi (F_* I_{e+1}(R)) \subseteq I_e(R)$ 
for every $\psi \colon F_*(R) \to R$ and $e>0$. 

The ideal $\cP(R) = \bigcap_{e >0} I_e(R)$ is prime and it is called the \emph{splitting prime} of $R$ \cite{AE}. Moreover, $\dim(R/\cP(R)) = \sdim(R)$ is called \emph{splitting dimension} of $R$ \cite{AE, BlickleSchwedeTucker}. In this setup, the three conditions are equivalent: $R$ is strongly F-regular, $\cP(R) = 0$, and $\sdim(R)=\dim(R)$.

\begin{definition}[{\cite{KarlCentersFpurity}}] \label{def: compatible}
Suppose that $R$ is an F-pure ring.
 An ideal $J \subseteq R$ is said to be \emph{compatible} if 
$\phi(F^e_*(J)) \subseteq J$
 for all $e >0$ and all $\phi \in \Hom_R(F^e_*(R),R)$. 
\end{definition}

We recall that  $\cP(R) $ is the largest compatible proper ideal of $R$ \cite{AE}. We also collect in the following remark a useful characterization of compatible ideals, making use Fedder's Criterion (Theorem~\ref{FedderCriterion}).

\begin{remark} \label{compatibleFedder}
Let $(S,\n)$ be a F-finite regular local ring, $I\subseteq S$ be an ideal, and $R=S/I$. Let $J\subseteq R$ be an ideal, and let $\widetilde{J}$ denote its pullback to $S$.
We have that $J$ is compatible if and only if
$(I^{[p^e]} :_S I)\subseteq (\widetilde J^{[p^e]} :_S \widetilde J)$ for all $e>0$.
\end{remark}

\subsection{F-pure thresholds}

\begin{definition}[\cite{TakagiWatanabe}]\label{DefFptGlobalIdeal}
Let $R$ be an F-finite F-pure ring, and $I\subseteq R$ be an ideal.
\begin{itemize}
\item For $\lambda \in \R_{\geq 0}$ we say that $(R,I^\lambda)$ is F-pure if, for all $e\gg 0$, there exists $f\in I^{\lfloor \lambda(p^e-1) \rfloor }$ such that the $R$-linear map $R\to F^e_*(R)$ sending $1 \mapsto F^e_*(f)$ splits.
\item The F-pure threshold of $I$ in $R$ is defined by 
$$
\fpt(I,R)=\sup\{ \lambda\in \mathbb{R}_{\geq 0}\; |\; (R,I^\lambda) \hbox{ is F-pure}\}.
$$
If the ring is clear from the context, we only write $\fpt(I)$.
If $R$ is local and $I$ is its maximal ideal, the F-pure threshold $\fpt(\m)$ is denoted by $\fpt(R)$.
\end{itemize}
\end{definition}

Note that, if $R$ is not F-pure, then $\fpt(I)= -\infty$ for any ideal $I$ if $R$. We now give a definition of the F-threshold for a local ring which is equivalent to Definition \ref{DefFptGlobalIdeal}. First we need to recall some notions.

\begin{definition}
Let $R$ be a ring, $I$ be an ideal of $R$, and $M$ be an $R$-module. 
Then the Loewy length of $M$ with respect to $I$ is defined as 
\[
\loewy_I (M) \coloneqq \inf \{n \mid I^nM = 0\}.
\]
In the case when $R$ is local and $I$ is its maximal ideal, we 
 omit the index and simply write $\loewy (M)$. 
\end{definition}

For an ideal $I\subseteq \m$, we let
\[ 
b_I(p^e,R) \coloneqq  \max \left\{t \in \NN \mid I^t \not\subseteq I_e(R) \right\}.
\]
Note that $I^t \subseteq \m^{[p^e]} \subseteq I_e(R)$ for all $t \gg 0$, therefore $b_I(p^e,R)$ is well-defined. Also, note that $b_I(p^e,R) = \loewy_I (R/I_e(R))-1$. If the ring is clear from the context, we only write $b_I(p^e)$, and if $R$ is local and $I$ is its maximal ideal, we only write $b(p^e)$.

The proof of the following result is analogous to that for standard graded $\kk$-algebras \cite[Proposition 3.10]{DSNBFpurity}. For this reason, we omit it.

\begin{proposition}[\cite{TakagiWatanabe,DSNBFpurity}]
Let $(R,\m)$ be an F-finite F-pure ring, and $I\subseteq R$ an ideal.  Then
\[
\fpt(I)=\lim\limits_{e\to \infty} \frac{b_I(p^e) }{p^e}.
\]
\end{proposition}

\begin{definition}[{\cite[Theorem A]{DSNBP}, see also \cite{HMTW}}]
Let $R$ be a ring of prime characteristic $p$. Given ideals $I,J$ inside $R$ satisfying $I \subseteq \sqrt{J}$, and $e \geq 0$, we let
\[ 
\nu^J_I(p^e)=\max\{t \in \NN \mid I^t \not\subseteq J^{[p^e]} \}.
\]
The F-threshold of $I$ in $J$ is defined as
$$
c^J(I)=\lim\limits_{e\to \infty}\frac{\nu^J_I(p^e)}{p^e}.
$$ 
\end{definition}

\begin{remark}\label{rem: fpt and ft inequality}
Since $I_e(R)^{[p^{e'}]} \subseteq I_{e+e'}(R)$, 
it follows that $\nu^{I_e(R)}_I(p^{e'}) \geq b_I (p^{e+e'})$, so that $p^e\fpt(I) \leq \Ft^{I_e(R)}(I)$. In particular, one gets $\fpt(I) \leq \Ft^{\m}(I)$ by choosing $e=0$. 
\end{remark}

We use Remark~\ref{rem: fpt and ft inequality} to derive a convergence 
estimate for F-pure threshold from a convergence estimate for F-threshold.

\begin{lemma}[{\cite[Lemma~3.3]{DSNBP}}]\label{lem: F-threshold convergence}
Let $R$ be a ring of prime characteristic, and $I,J$ be ideals of $R$ satisfying $I \subseteq \sqrt{J}$.
If $I$ can be generated by $\mu$ elements, then
for all integers $e_1, e_2 \geq 0$ one has
\[
\frac{\nu_I^J (p^{e_1 + e_2})}{p^{e_1 + e_2}} - \frac{\nu_I^J(p^{e_1})}{p^{e_1}}\leq \frac{\mu}{p^{e_1}}.
\]
\end{lemma}

\begin{proposition}\label{prop: fpt convergence}
Let $(R, \m)$ be an F-finite F-pure local ring of prime characteristic.
Then for all $e \geq 0$ we have
\[
0 \leq \fpt(R) - \frac{b (p^e)}{p^{e}} \leq \frac{\edim (R)}{p^e}.
\]
\end{proposition}
\begin{proof}
We apply Lemma~\ref{lem: F-threshold convergence} to $I=\m$, $J=I_e(R)$, and $e_1=0$ to obtain that 
\[ \Ft^{I_e(R)} (\m) \leq b(p^e) + \edim (R).
\]
We use Remark \ref{rem: fpt and ft inequality} to get
\[
p^e \fpt (R) \leq \Ft^{I_e(R)} (\m) \leq b(p^e) + \edim (R).
\]

Since $R$ is F-pure there a surjection $\psi \colon F_*^1 (R) \to R$.
By definition, if $\m^{tp} \subseteq I_{e+1}(R)$ holds for some $t>0$,
then 
\[
I^t = \psi (F_* ((I^t)^{[p]})) \subseteq \psi(F_*(I^{tp})) \subseteq \psi (F_* (I_{e+1}(R))) \subseteq I_e(R).
\]
This yields the inequality $pb (p^e) \leq b (p^{e+1})$ and, iterating this argument, one gets $p^{e'}b(p^e) \leq b(p^{e+e'})$ for all $e' >0$. Dividing by $p^{e+e'}$ and taking limits as $e' \to \infty$ gives that $b (p^e)/p^e \leq \fpt(R)$, and the assertion follows. 
\end{proof}

\subsection{Definition and first properties of the defect of the F-pure threshold}

We formally introduce the main invariant studied in this article.
\begin{definition} 
Let $(R,\m)$ be an F-finite local ring. The {\it defect of the F-pure threshold} is
\[
\dfpt(R) = \dim(R) -\fpt(R).
\]
\end{definition}

We start by restarting two results of Takagi and Watanabe \cite{TakagiWatanabe}, which show that the defect of the F-pure threshold detects important properties of a local ring.

\begin{proposition}[{\cite[Proposition~2.6]{TakagiWatanabe}}]\label{prop: TW reduction}
Let $(R, \m)$ be an F-finite F-pure local ring with infinite residue field. 
Suppose that $\dim (R) > 0$ and let $J$ be a minimal reduction of $\m$. 
If $k$ is an integer such that $\dfpt (R) < k$, then 
we have an inclusion $\m^k \subseteq J$.  
\end{proposition}

\begin{proposition}[{\cite[Theorem 2.7]{TakagiWatanabe}}] \label{prop TW regular}
Let $(R, \m)$ be an F-pure F-finite local ring of dimension $d > 0$. Then the following are equivalent:
\begin{enumerate}
\item $\dfpt(R)  <1$,
\item $\dfpt(R) =0$,
\item $R$ is regular.
\end{enumerate}
\end{proposition}


\begin{lemma}\label{lem: compatible}
Let $(R,\m)$ be an F-finite and F-pure local ring. If $J\subseteq R$ is a compatible ideal, then 
$
\fpt(R)\leq \fpt(R/J).
$
In particular, $\dfpt(R)\geq \dim(R) - \sdim(R) \geq 
\dim(R)-\depth(R)$, the \emph{Cohen-Macaulay defect} of $R$.
\end{lemma}
\begin{proof}
The first assertion follows immediately from Definitions~\ref{def: compatible} and~\ref{DefFptGlobalIdeal}. For the second assertion, we let $J$ be the splitting prime of $R$ to get that 
$\fpt(R)\leq \sdim(R) \leq \depth(R),$
where the second inequality is obtained by combining work of Yao \cite[Lemma~2.2]{YaoFrep} with a characterization of the splitting dimension  \cite[Corollary~4.3]{BlickleSchwedeTucker} (see also \cite[Theorem~4.8]{AE}).
\end{proof}

A strongly F-regular local ring must satisfy $\dfpt (R) < \dim(R)$. 
The following proposition provides a partial converse and, furthermore, shows that rings with a small defect have mild singularities.

\begin{proposition}\label{prop: condition F-regular}
Let $(R,\m)$ be an F-finite and F-pure local ring. If there is a positive integer $k$ such that $\dfpt (R) < k + 1$ 
and 
$R_\p$ is strongly F-regular whenever $\dim (R_\p) \leq k$, then $R$ is strongly F-regular. In particular, if $\dfpt (R) < 2$ and $R$ is normal, then $R$ is strongly F-regular. 
\end{proposition}
\begin{proof}
Let $\cP$ denote the splitting prime of $R$. By Lemma~\ref{lem: compatible} 
\[
k + 1 > \dfpt(R) \geq \dim (R) - \sdim(R), 
\]
so that $\sdim (R) \geq \dim (R) - k$. 
Thus $\height (\cP ) \leq k$ and, since $R_{\cP}$ is strongly F-regular by assumption, this shows that $R_{\cP}$ is a field  \cite[Corollary~4.6]{AE}. Hence, $\cP$ is a minimal prime of $R$ and, in fact, it must be the only minimal prime because all zero-divisors are necessarily contained in $\cP$. Therefore  $\sdim (R) = \dim (R)$, and $R$ is strongly F-regular  \cite[Theorem~4.8]{AE}. The second assertion follows directly from the first.
\end{proof}

\begin{example}
Let $\kk$ be an F-finite field, and $R=\kk\ps{x,y,z}/(xy)$. Then $\dfpt(R)=1$, but $R$ is not strongly F-regular. This example shows that
the assumption on the F-regular locus is necessary, as $R$ fails to be strongly F-regular at the height one prime $(x,y)$. 
\end{example}

Even if not strongly F-regular, a local ring $(R, \m)$ with small $\dfpt (R)$ still has nice properties. For example, we may use the defect of the F-pure threshold to give a bound for the Hilbert-Samuel multiplicity of F-pure rings. This follows the ideas of previous work by Huneke and Watanabe \cite{HW15}. If $(R,\m)$ is local, we let $\ecodim(R) = \edim(R) - \dim(R)$ be the embedding codimension of $R$, and 
\[
\e(R) = \lim\limits_{n \to \infty} \frac{d!\ell_R(R/\m^{n})}{n^{d}}
\]
be its Hilbert-Samuel multiplicity, where $d=\dim(R)$ and $\ell_R(-)$ denotes the length of an $R$-module.

\begin{proposition} 
Let $(R,\m)$ be an F-finite F-pure local ring.
Then
$$
\e(R)\leq \binom{\ecodim(R)+ \lceil \dfpt(R)\rceil}{\lceil \dfpt(R)\rceil}.
$$
\end{proposition}
\begin{proof}
Let $(S,\n)$ be an F-finite regular ring mapping onto $R$, and $I \subseteq S$ be an ideal such that $R=S/I$. We can assume without loss of generality that $\dim(S)=\edim(R)$, that is, $I \subseteq \n^2$. With this assumption, we have that $\ecodim(R) = \height_S(I)$.
Let $\kk$ be a coefficient field for $\widehat S$,
and let $T = S \cotimes{\kk} \kk(t)$, where $t$ is a variable. Note that $S \to  T$ is a flat local morphism, and that the maximal ideal of $S$ extends to the maximal ideal of $T$. Let $A=T/IT$. Then $I_e(R)A = I_e(A)$ for all $e$ \cite[Lemma 3.8]{AE}, and thus 
$\fpt(R)=\fpt(A)$. Moreover, because $\dim(R)=\dim(A)$, we have that $\dfpt(R)=\dfpt(A)$. By comparing Hilbert functions, we also deduce that
$\e(R)=\e(A)$.
Therefore, we may assume that $R$ is a complete local ring that  contains an infinite field.

Let $J$ be a minimal reduction of $\m$. We have that 
$
\e(R)\leq \ell_R(R/J)=\dim_\kk(R/J)
$ 
\cite[Corollary 4.7.11]{BrHe}.
Let $d=\dim(R)$ and $t= \height_S(I)$.
Let $x_1,\ldots,x_d\in R$ be minimal generators of $J$ and $y_1,\ldots,y_t\in R$ be such $\m=(x_1,\ldots,x_d,y_1,\ldots,y_t)$.
By Proposition~\ref{prop: TW reduction} $\m^{\lceil\dfpt(R)+1\rceil}\subseteq J$, so the set $\left\{ y_1^{\alpha_1}\cdots y_t^{\alpha_t} \mid \alpha_1, \ldots, \alpha_r \in \mathbb{Z}_{\geq 0}, \sum_{i = 1}^t \alpha_i \leq \lceil\dfpt(R)\rceil\right\}$  generates $R/J$ as a $\kk$-vector space. Hence, we conclude that
\[
\e(R)\leq \ell_R(R/J)=\dim_\kk(R/J)\leq \binom{t+\lceil\dfpt(R)\rceil}{\lceil\dfpt(R)\rceil}.
\qedhere
\]
\end{proof}

\section{Differential operators and the defect of the F-pure threshold}\label{sect: differential}

\subsection{Differential operators and differential powers}\label{ssect: differential}

\begin{definition}[\cite{EGA}]
Let $A$ be a ring, and $S$ be an $A$-algebra. We define the $A$-linear \emph{differential operators} of order at most $n \in \NN$ on $S$ inductively by setting
\begin{enumerate}
	\item $D^0_{S|A} = \Hom_S(S,S) \subseteq \Hom_{\ZZ}(S,S)$,
	\item $D^n_{S|A}=\{ \delta \in \Hom_{A}(S,S) \ | \ \delta \circ \phi - \phi \circ \delta \in D^{n-1}_{S|A} \ \text{for all} \ \phi\in D^0_{S}\}$.
\end{enumerate}

We call $D_{S|A}=\bigcup_{n \geq 0} D^n_{S|A}$ the \emph{ring of $A$-linear differential operators} on $S$. If $A=\ZZ$, we write $D^n_S$ and $D_{S}$ in place of $D^n_{S|\ZZ}$ and $D_{S|\ZZ}$, respectively.
\end{definition}

We now present a description of differential operators in 
prime characteristic due to Yekutieli \cite{Ye}. 
We also  refer to work of Quinlan-Gallego \cite[Proposition 2.4]{QG-Fmodules} for another proof of this statement.

\begin{theorem}[{\cite{Ye}}] \label{thm:Yek}
If $S$ is an F-finite ring of characteristic $p>0$, then
\[
D_{S}=\bigcup_{e\in \NN} D^{(e)}_S,
\]
where $D_S^{(e)}= \Hom_{S^{p^e}}(S,S)$.
\end{theorem}

\begin{remark} \label{rem: perfect K} Let $S$ be a ring, and $\kk \subseteq S$ be a field. By definition, we always have $D^n_{S|\kk} \subseteq D^n_S$ for every $n \in \NN$. In addition, if $\kk$ is perfect and $S$ is F-finite, we get equality. In fact, let $\delta \in D^n_S$, $\lambda \in \kk$ and $s \in S$. By Theorem \ref{thm:Yek} there exists $e \in \NN$ such that $\delta \in D^{(e)}_S$. Since $\kk$ is perfect there exists $\nu \in \kk$ such that $\lambda = \nu^{p^e} \in S^{p^e}$, and therefore $\delta(\lambda s) = \delta(\nu^{p^e}s) = \nu^{p^e}\delta(s) = \lambda \delta(s)$. It follows that $\delta \in D^n_{S|\kk}$, as claimed.
\end{remark}

\begin{notation} \label{NotationD(n,e)}
Let $A$ be a ring, and $S$ be an F-finite $A$-algebra. 
For all integers $n,e \geq 0$ we let
$$D^{(n, p^e)}_{S|A}=D_{S|A}^{n} \cap D_{S}^{(e)}.$$
Note that for all $e>0$ we have
\begin{itemize}
	\item $D^{(0,p^e)}_{S|A} = \Hom_S(S,S) \subseteq D^{(e)}_S$,
	\item  $D^{(n,p^e)}_{S|A}=\{ \delta \in D_S^{(e)} \cap \Hom_A(S,S)\ | \ \delta \circ \phi - \phi \circ \delta \in D^{(n-1,p^e)}_{S|A} \ \text{for all} \ \phi\in D^{(0,p^e)}_{S|A}\}$
\end{itemize}
We note that $S\cong \Hom_S(S,S)$, and by abusing notation, we write $s\in S$ instead of the map given by multiplication by $s$.
If $A=\ZZ$, we write $D^{(n,p^e)}_S$ instead of $D^{(n,p^e)}_{S|\ZZ}$.
\end{notation}

\begin{remark}\label{RemCommutator}
Let $\delta\in D^{(e)}_S$. Then $\delta \in D^{(n, p^e)}_S$ if and only if 
$$
[[\ldots [[\delta, s_0],s_1],\ldots],s_n]=0
$$
for every $s_0,\ldots,s_n\in S$, where $[-,-]$ denotes the commutator.
\end{remark}

\begin{example}\label{ex: divided powers}
In the simplest case of $S = A[x_1,\ldots,x_d]$, the ring
$D_{S|A}$ is generated as an $S$-module by the so-called \emph{divided powers} differential operators: 
for any tuple $\underline{\alpha} \in \NN^d$ we define 
\[
\partial^{(\underline{\alpha})} (x_1^{\beta_1} \cdots x_d^{\beta_d})
\coloneqq 
\begin{cases}
{\beta_1 \choose \alpha_1}\cdots {\beta_d \choose \alpha_d} x_1^{\beta_1-\alpha_1}\cdots x_d^{\beta_d-\alpha_d}
& \text{if } \beta_i \geq \alpha_i \text{ for all } i \\
0 & \text{otherwise}
\end{cases}.
\]
In fact, $D_{S|A}^n$ and $D^{(n, p^e)}_{S|A}$ are finitely generated free $S$-module: a basis of $D_{S|A}^n$ is given by $\partial^{(\underline{\alpha})}$
for $\underline{\alpha} \in \NN^d$ such that $\alpha_1 + \cdots + \alpha_d \leq n$, while a basis for $D^{(n, p^e)}_{S|A}$ 
is given by further imposing that $\alpha_i < p^e$ for all $i$.
\end{example}

\begin{remark}\label{PPFinite}
Let $A$ be a ring, and $S$ be an F-finite $A$-algebra. Let $\Delta_{S|A}$ be the kernel of the multiplication map $S\otimes_A S\to S$.
We note that 
$$
\frac{S\otimes_\ZZ S}{\Delta^{[p^e]}_{S|\ZZ}}\cong S\otimes_{S^{p^e}}S,
$$
and so, 
$$ 
\frac{S\otimes_\ZZ S}{\Delta^{n+1}_{S|\ZZ}+\Delta^{[p^e]}_{S|\ZZ}} \cong \frac{S \otimes_{S^{p^e}} S}{\Delta_{S|S^{p^e}}^{n+1}}
$$ is a finitely generated $S$-module for all $e,n \geq 0$. Furthermore, we have that
\begin{align*}
\Hom_S\left( \frac{S\otimes_\ZZ S}{\Delta^{n+1}_{S|\ZZ}+\Delta^{[p^e]}_{S|\ZZ}}, S\right) & \cong \Ann_S(\Delta^{n+1}_{S|\ZZ} + \Delta^{[p^e]}_{S|\ZZ}) \\
& \cong \Ann_{S} (\Delta^{n+1}_{S|\ZZ}) \cap \Ann_{S} (\Delta^{[p^e]}_{S|\ZZ})  \\
& \cong \Hom_S\left( \frac{S\otimes_\ZZ S}{\Delta^{n+1}_{S|\ZZ}}, S\right)  \cap \Hom_S\left( \frac{S\otimes_\ZZ S}{\Delta^{[p^e]}_{S|\ZZ}}, S\right) \\
& \cong D^n_S \cap D_S^{(e)}\\
& = D^{(n,p^e)}_{S},
\end{align*}
where the intersection in the third line comes from viewing the dual of a quotient of $S \otimes_\ZZ S$ as a submodule of $\Hom_S(S \otimes_\ZZ S,S) \cong \Hom_\ZZ(S,S)$. It follows that $D^{(n,p^e)}_{S}$ is a finitely generated $S$-module.
\end{remark}

\begin{lemma}\label{LemmaLocalizationD^{(n,e)}}
Let $S$ be an F-finite ring and $W\subseteq S$ be a multiplicative system. Then  $W^{-1}D^{(n,p^e)}_S \cong D^{(n,p^e)}_{W^{-1} S}$ for all $n,e \geq 0$.
\end{lemma}
\begin{proof}
We have that
\begin{align*}
W^{-1} D^{(n,p^e)}_S
& \cong D^{(n,p^e)}_S\otimes_R W^{-1} S\\
& \cong 
\Hom_S\left( \frac{S\otimes_\ZZ S}{\Delta^{n+1}_{S|\ZZ}+\Delta^{[p^e]}_{S|\ZZ}}, S\right)\otimes_S W^{-1} S\\
& \cong \Hom_{W^{-1} S}\left( \frac{S\otimes_\ZZ S}{\Delta^{n+1}_{S|\ZZ}+\Delta^{[p^e]}_{S|\ZZ}} \otimes_S W^{-1}  , W^{-1} S \right)\\
&\cong \Hom_{W^{-1} S}\left( \frac{S\otimes_{S^{p^e}} S}{\Delta^{n+1}_{S|S^{p^e}}} \otimes_S W^{-1} S , W^{-1} S \right)\\
&\cong \Hom_{W^{-1} S}\left( \frac{W^{-1} S\otimes_{W^{-1} S^{p^e}} W^{-1} S}{\Delta^{n+1}_{W^{-1} S |W^{-1} S^{p^e}}} , W^{-1} S\right)\\
&\cong \Hom_{W^{-1}S}\left(\frac{W^{-1}S \otimes_\ZZ W^{-1}S}{\Delta_{W^{-1}S|\ZZ}^{n+1} + \Delta_{W^{-1}S|\ZZ}^{[p^e]}},W^{-1}S\right) \\
& \cong D^{(n,p^e)}_{W^{-1} S},
\end{align*}
where we use the finite generation results established in Remark \ref{PPFinite} multiple times, as well as properties of principal parts with localization \cite[Proposition 2.16]{BJNB} and the fact that $W^{-1}S^{p^e} = (W^{-1}S)^{p^e}$.
\end{proof}

\begin{lemma}\label{LemmaCompletion D^{(n,e)}}
Let  $(S,\m,\kk)$ be an F-finite regular local ring. Then  $\widehat{D^{(n,p^e)}_S} \cong D^{(n,p^e)}_{\widehat{S}}$.
\end{lemma}
\begin{proof}
Let $\widehat{S}$ denote the $\m$-adic completion of $S$. We have that
\begin{align*}
\widehat{D^{(n,p^e)}_S} 
& \cong D^{(n,p^e)}_S\otimes_S \widehat{S}\\
& \cong \Hom_S\left( \frac{S\otimes_\ZZ S}{\Delta^{n+1}_{S|\ZZ}+\Delta^{[p^e]}_{S|\ZZ}}, S\right)\otimes_S \widehat{S}\\
& \cong \Hom_S\left( \frac{S\otimes_\ZZ S}{\Delta^{n+1}_{S|\ZZ}+\Delta^{[p^e]}_{S|\ZZ}} \otimes_S \widehat{S} , \widehat{S} \right)\\
&\cong \Hom_{\widehat{S}}\left( \frac{S\otimes_{S^{p^e}} S}{\Delta^{n+1}_{S|S^{p^e}}} \otimes_S \widehat{S} , \widehat{S} \right)\\
&\cong \Hom_{\widehat{S}}\left( \frac{\widehat{S}\otimes_{\widehat{S}^{p^e}} \widehat{S}}{\Delta^{n+1}_{\widehat{S}|\widehat{S}^{p^e}}} , \widehat{S} \right)\\
& \cong \Hom_{\widehat{S}}\left( \frac{\widehat{S}\otimes_{\ZZ} \widehat{S}}{\Delta^{n+1}_{\widehat{S}|\ZZ} + \Delta^{[p^e]}_{\widehat{S}|\ZZ}} , \widehat{S} \right) \\
&\cong D^{(n,p^e)}_{\widehat{S}},
\end{align*}
where we use the finite generation results obtained in Remark \ref{PPFinite} several times, together with the fact that $\widehat{S^{p^e}}\cong \widehat{S}^{p^e}$.
\end{proof}

\begin{definition}
Let $\kk$ be an F-finite field. We say that a set $\Lambda=\{\lambda_1,\ldots,\lambda_a\}\subseteq \kk\setminus \kk^p$ is a $p$-basis of $\kk$ if $\kk=\kk^p[\Lambda]$ and $[\kk:\kk^p]=p^a$.
\end{definition}

\begin{example}
Let $\kk$ be an F-finite field, $S=\kk\ps{x_1,\ldots, x_d}$, and $\Lambda=\{\lambda_1,\ldots,\lambda_a\}$ be a $p$-basis of $\kk$.
For $e\in \NN$ and  $\alpha\in\NN^a$ such that $0\leq \alpha_i\leq p^e-1$, we let $\lambda^\alpha=\lambda^{\alpha_1}_1\cdots \lambda^{\alpha_a}_a$. For $\beta \in \NN^d$ we also let $x^\beta=x_1^{\beta_1} \cdots x_d^{\beta_d}$. We note that $\{\lambda^\alpha x^\beta\; |\; 0\leq \alpha_i,\beta_i\leq p^e -1\}$ is a basis for $S$ as an $S^{p^e}$-module. Given $\gamma\in \ZZ^a_{\geq 0}$, we write
$(1\otimes\lambda-\lambda\otimes 1)^\gamma$ for
$\prod^a_{i=1}(1\otimes\lambda_i-\lambda_i\otimes 1)^{\gamma_i}$
Given $\theta\in \ZZ^d_{\geq 0}$, we write $(1\otimes x-x\otimes 1)^\theta$ for
$\prod^d_{i=1} (1\otimes x-x\otimes 1)^{\theta_i}$. Then, 
\begin{align*}
\frac{S\otimes_\ZZ S}{\Delta^{n+1}_{S|\ZZ}+\Delta^{[p^e]}_{R|\ZZ}} & \cong 
\frac{S\otimes_{S^{p^e}} S}{\Delta^{n+1}_{S|S^{p^e}}} \cong 
\frac{\bigoplus_{\substack{0\leq \alpha_i,\beta_j,\gamma_i,\theta_j\leq p^e-1}} S^{p^e}(\lambda^\alpha x^\beta\otimes \lambda^\gamma x^\theta)}{ (1\otimes \lambda_i  - \lambda_i \otimes 1,  1\otimes x_j-x_j\otimes 1)^{n+1} }\\
&\cong 
\frac{\bigoplus_{\substack{0\leq \alpha_i,\beta_j,\gamma_i,\theta_j\leq p^e-1}} S^{p^e} (\lambda^\alpha x^\beta\otimes 1 )
(1\otimes\lambda-\lambda\otimes 1)^\gamma (1\otimes x-x\otimes 1)^\theta}
{ (1\otimes \lambda_i  - \lambda_i \otimes 1,  1\otimes x_j-x_j\otimes 1)^{n+1} }\\
&\cong 
\bigoplus_{\substack{0\leq \alpha_1,\beta_j,\gamma_i,\theta_j\leq p^e-1 \\
|\gamma|+|\theta|\leq n}} S^{p^e} (\lambda^\alpha x^\beta\otimes 1 )
(1\otimes\lambda-\lambda\otimes 1)^\gamma (1\otimes x-x\otimes 1)^\theta\\
&\cong 
\bigoplus_{\substack{0\leq \gamma_i,\theta_j\leq p^e-1 \\
|\gamma|+|\theta|\leq n}} S
(1\otimes\lambda-\lambda\otimes 1)^\gamma (1\otimes x-x\otimes 1)^\theta,
\end{align*}
It follows that $ \frac{S\otimes_\ZZ S}{\Delta^{n+1}_{S|\ZZ}+\Delta^{[p^e]}_{S|\ZZ}}$ is a free $S$-module, and hence so is $D^{(n,p^e)}_{S}$.
\end{example}

We now introduce two notions of differential powers of a given ideal. The first one has been considered before in the literature, the second one is new.

\begin{definition}
Let $S$ ne a ring, and $I\subseteq S$ be an ideal. For $n,e\in\NN$
\begin{itemize}
\item we define the $n$-th differential power of $I$ as
$I^{\{n\}} \coloneqq \{s \in S \mid \delta (s) \in I \text{ for all } \delta \in D_{S}^{n-1}\}$ \cite{SurveySymbPowers};
\item we define the $n$-th differential power of level $e$ of $I$ as
$$I^{\{n,p^e\}} \coloneqq \{s \in S \mid \delta (s) \in I \text{ for all } \delta \in D_{S}^{(n-1,p^e)}\}.$$
\end{itemize} 
\end{definition}

\begin{proposition}\label{prop: dif pow properties}
Let $S$ be a F-finite ring and $I \subseteq S$ be an ideal.
Then $I^{\{n, p^e\}}$ is an ideal. Moreover, if $\p$ is a prime ideal, then  $\p^{\{n, p^e\}}$ is a $\p$-primary ideal containing $\p^n + \frpe{\p}$.
\end{proposition}
\begin{proof}
First, $I^{\{n, p^e\}}$ is closed under addition.
We now show that it is closed under multiplication by induction on $n$.
Since $D^0_{S}=S$, we have that $I^{\{1, p^e\}} = I$. 
For $n \geq 1$, let $f \in I^{\{n+1, p^e\}}$ and $s \in S$. 
Then for any $\delta \in D_{S}^{n} \cap D_{S}^{(e)}$, we may compute 
\[
\delta (sf) = [\delta, s](f) + s\delta(f) \in I,
\]
where we use that $[\delta, s] \in  D_{S}^{n-1}\cap D_{S}^{(e)}$ and $f \in I^{n+1,p^e} \subseteq I^{n,p^e}$.
 
We now show that $\p^{\{n, p^e\}}$ is $\p$-primary, proceeding again by induction on $n$. We already observed that $\p^{\{1, p^e\}} = \p$.
Suppose that $f \in \p$, $s \notin \p$, and $sf \in \p^{\{n+1, p^e\}}$. 
By induction, $f \in \p^{\{n, p^e\}}$.
For any $\delta \in D_{S}^{n} \cap D_{S}^{(e)}$ we have 
that $[\delta, s] \in D_{S}^{n-1} \cap D_{S}^{(e)}$, so 
\[
s\delta(f) = \delta (sf) - [\delta, s](f) \in \p.
\]
Since $\p$ is prime, $\delta(f) \in \p$ and the claim follows. 

Lastly, since $\p^{\{n,p^e\}}$ is $\p$-primary, in order to show the containment it suffices to prove that $(\p^n + \frpe{\p})S_\p = \p^{\{n, p^e\}}S_\p$. 
We note that  
$$
\p^n\subseteq \{s \in S \mid \delta (s) \in \p \text{ for all } \delta \in D_{S}^{n-1}\}= \p^{\{n\}}\subseteq \p^{\{n, p^e\}}
$$ 
from previous work on differential powers \cite[Proposition 2.5]{SurveySymbPowers}.
In addition,   
$$
\p^{[p^e]}\subseteq \{s \in S \mid \delta (s) \in \p \text{ for all } \delta \in D_{S}^{(e)}\}\subseteq \p^{\{n, p^e\}}.
$$
It follows that $(\p^n + \frpe{\p})S_\p \subseteq  \p^{\{n, p^e\} }S_\p$, as claimed.
\end{proof}

\begin{lemma}\label{lemma: dif pow maximal series}
Let $S=\kk\ps{x_1,\ldots,x_d}$ be a power series ring over an F-finite field $\kk$, and $\m$ be its maximal ideal.
Then  $\m^{\{n, p^e\}}=\m^n + \frpe{\m}$ for all $n,e \geq 0$.
\end{lemma}
\begin{proof}
By Proposition~\ref{prop: dif pow properties} it suffices to show that $\m^{\{n, p^e\}}\subseteq\m^n + \frpe{\m}$.
Let $\prec$ denote the degree-lexicographic order on monomials of $S$. Let $f\not\in \m^n + \frpe{\m}$, and consider the monomial
\[
x^\alpha \coloneqq x_1^{\alpha_1}\cdots x_{d}^{\alpha_d}
= \min_\prec \{x^\beta \in \Supp(f),  x^{\beta} \notin  \m^n+\m^{[p^e]}\}.
\]
By the way it is defined, we must have $\alpha_1 + \cdots + \alpha_t \leq n-1$ and $\alpha_i \leq p^e-1$ for every $i$. Thus, the divided powers operator 
$\delta \coloneqq \partial^{(\alpha_1, \ldots, \alpha_t)}$ 
is such that  $\delta \in  D^{(n,p^e)}_S$,  $\delta(x^{\alpha})=1$, and $\delta(x^\beta)\in\m$ for every monomial $\alpha\leq \beta$ such that $\alpha \neq \beta$. It follows that $\delta(f)\not\in\m$, and hence $f\not\in \m^{\{n, p^e\}}$.
\end{proof}

\begin{lemma}\label{LemmaDiffPowersLocalization}
Let $S$ be an F-finite domain, and $W \subseteq S$ be a multiplicative system. Let $I\subseteq S$ be an ideal such that $W^{-1} I \cap S=  I$. For all $n,e \geq 0$ we have that $W^{-1}(I^{\{n,p^e\}})=(W^{-1}I)^{\{n,p^e\}}$.
\end{lemma}
\begin{proof}
Let  $f\in I^{\{n,p^e\}}$, so that $\delta (f)\in I$ for every $\delta \in D^{(n-1,e)}_S$.
From this condition it follows that $\delta' (f)\in W^{-1}I$ for every $\delta' \in W^{-1} D^{(n-1,e)}_S$. 
Since $W^{-1} D^{(n-1,e)}_S=D^{(n-1,e)}_{W^{-1}S}$ by Lemma \ref{LemmaLocalizationD^{(n,e)}}, we conclude that $W^{-1}(I^{\{n,p^e\}})\subseteq (W^{-1}I)^{\{n,p^e\}}.$

Conversely, let $f\in (W^{-1}I)^{\{n,p^e\}}\cap S$.  Then  $\delta (f)\in W^{-1} I$ for every $\delta \in D^{(n-1,e)}_S\subseteq W^{-1} D^{(n-1,e)}_S=D^{(n-1,e)}_{W^{-1}S}$. It follows that $\delta (f)\in W^{-1} I\bigcap S=I$, and hence $f\in I^{\{n,p^e\}}$. We conclude that
$(W^{-1}I)^{\{n,p^e\}}\cap S\subseteq I^{(n,p^e)}$, from which it follows that $ (W^{-1}I)^{\{n,p^e\}}\subseteq W^{-1}I^{\{n,p^e\}}$.
\end{proof}

\begin{lemma}\label{LemmaDiffPowersCompletion}
Let $(S,\m)$ be an F-finite regular local  ring and $I\subseteq S$ an ideal.
Then $(I\widehat{S})^{\{n,e\}}=I^{\{n,p^e\}}\widehat{S}$ for all $n,e \geq 0$.
\end{lemma}
\begin{proof}
Let  $f\in I^{\{n,p^e\}}$. Then $\delta (f)\in I$ for every $\delta \in D^{(n-1,p^e)}_S$, and thus $\phi (f)\in I\widehat{S}$ for every $\phi \in  D^{(n-1,p^e)}_S \otimes_S \widehat{S}$. Since $S$ is F-finite, we have that $D^{(n-1,p^e)}_S \subseteq D^{(e)}_S$ is finitely generated. In particular, $D^{(n-1,p^e)}_S \otimes_S \widehat{S} \cong \widehat{D^{(n-1,p^e)}_S}$, and by 
Lemma \ref{LemmaCompletion D^{(n,e)}} the latter is isomorphic to $D^{(n-1,p^e)}_{\widehat{S}}$. We therefore conclude that $I^{\{n,p^e\}}\widehat{S}\subseteq (I\widehat{S})^{(n,p^e)}.$

Conversely, assume that $f\not\in I^{\{n,p^e\}}\widehat{S}$. Using the isomorphism described above, we conclude that there exists 
$\varphi \in   D^{(n-1,e)}_S\otimes_S \widehat{S}$ such that
$\varphi (f)\not\in  I\widehat{S}$. In particular, there must exist $\rho\in D^{(n-1,e)}_S$ such that $\rho(f)\not\in  I$.
Hence, $f\not\in (I\widehat{S})^{\{n,p^e\}}$, and therefore 
$(I\widehat{S})^{\{n,p^e\}}\subseteq I^{\{n,p^e\}}\widehat{S}.$
\end{proof}

\begin{proposition}\label{prop:Analogue of ZN}
Let $S$ be a F-finite regular ring, and $\p\subseteq S$ be a prime ideal.
Then $\p^{\{n,p^e\}}$ is the $\p$-primary component of $\p^n + \p^{[p^e]}$.
\end{proposition}
\begin{proof}
It suffices to show that $\p^{\{n,p^e\}}S_\p=(\p^{[p^e]}+\p^n)S_\p$, because $\p^{\{n,p^e\}}$
is $\p$-primary by Lemma \ref{prop: dif pow properties}. Let $A=\widehat{S_\p}$ be the $\p S_\p$-adic completion of $S_\p$. 
By Lemma \ref{lemma: dif pow maximal series}, we have that
$(\p A)^{\{n,p^e\}}=(\p A)^{[p^e]}+(\p A)^n$.
It follows that
$(\p A)^{\{n,p^e\}}=(\p S_\p)^{\{n,p^e\}}A$ and 
$(\p A)^{[p^e]}+(\p A)^n=((\p S_\p)^{[p^e]}+(\p S_\p)^n)A$
by Lemma \ref{LemmaDiffPowersCompletion}, and then
$(\p S_\p)^{\{n,p^e\}}=(\p S_\p)^{[p^e]}+(\p S_\p)^n$.
Thus, we get 
$\p^{\{n,p^e\}}S_\p=(\p S_\p)^{\{n,p^e\}}$ by Lemma \ref{LemmaDiffPowersLocalization} and $(\p^{[p^e]}+\p^n)S_\p=(\p S_\p)^{[p^e]}+(\p S_\p)^n$. 
Finally, we conclude that
$\p^{\{n,p^e\}}S_\p=(\p^{[p^e]}+\p^n)S_\p$, as desired.
\end{proof}

\subsection{Formulas for F-pure thresholds and differential powers}\label{sect: main formula}

Recall from the previous section that, for an ideal $I$ in a ring $S$, we have defined its $n$-th differential power of level $e$ as 
\[
I^{\{n,p^e\}} = \{s \in S \mid \delta(s) \in I \text{ for all } \delta \in D^{(n-1,p^e)}_S\}.
\]

\begin{definition}\label{def: theta}
Let $(S, \m)$ be an F-finite regular local ring.
For an ideal $I$ and a positive integer $e$, we define
\[
\Theta_e (I) = \max \{n \mid I^{[p^e]} :_S I \subseteq \m^{\{n,p^e\}}=\m^{[p^e]}+\m^n\}.
\]
\end{definition}

The following formula, relating the newly introduced invariants $\Theta_e$ and F-pure thresholds, is crucial throughout this article.

\begin{theorem}\label{thm: main formula}
Let $(S,\n)$ be an F-finite regular local ring, and $I\subseteq S$ an ideal such that $R=S/I$ is F-pure. Then
\[
\loewy_R (R/I_e(R)) + \Theta_e(I)  = \dim (S) (p^e - 1) + 1.
\]
In particular, 
\[
\lim_{e \to \infty} \frac{\Theta_e(I)}{p^{e}}  = \dim (S) - \fpt(R) = \dfpt (R) + \height (I).
\]
\end{theorem}
\begin{proof}
Since $S$ is a regular local ring, $\dim(S) = \dim(R) + \height(I)$. 
We may complete and assume that $S$ is a power series ring on $\mu \coloneqq \dim (S)$ variables. By a slight abuse of notation, we  still denote by $\m$ the maximal ideal of $R$. By definition we have that
\[
\loewy_R (R/I_e(R)) = \min \{t \mid \m^t \subseteq I_e(R)\} 
= \min \{t \mid \phi (F_*(\m^t)) \subseteq \m 
\text{ for all } \phi \in \Hom(F_*^e R, R)\}. 
\]
Via Fedder's characterization in Theorem \ref{FedderCriterion}, and using that $\m^{\{n,p^e\}} = \m^n+\m^{[p^e]}$ by Lemma \ref{lemma: dif pow maximal series}, this is equivalent to 
\[
\loewy_R (R/I_e(R)) = \min  \{t \mid \m^t (I^{[p^e]}:_SI) \subseteq \m^{[p^e]}\}.
\]

Since $S$ is regular,  $\mu (p^e - 1) + 1 = \loewy_S (S/\m^{[p^e]})$. 
Hence, 
\[
\m^{\mu (p^e - 1) + 1 - \Theta_e(I)} \cdot \left(I^{[p^e]}:_SI\right)  \subseteq 
\m^{\mu (p^e - 1) + 1 - \Theta_e(I)} (\m^{\Theta_e(I)} +\m^{[p^e]})\subseteq 
\m^{\mu (p^e - 1) + 1} +\m^{[p^e]}\subseteq  \m^{[p^e]},
\]
giving that $\loewy_R (R/I_e(R)) + \Theta_e(I) \leq  \mu (p^e - 1) + 1$.

We now show the opposite inequality.
Since $R$ is F-pure, we have that $I^{[p^e]}:I\not\subseteq \m^{[p^e]}$. 
This, together with the containment $I^{[p^e]}:I\subseteq \m^{\Theta_e(I)} +\m^{[p^e]}$ implies that there exists
$f_e \in \left(I^{[p^e]}:_S I\right)\smallsetminus \m^{[p^e]}$ such that 
$\Theta_e(I) = \max \{t \mid f_e \in \m^t + \m^{[p^e]} \}$.
Hence there is a monomial $x_1^{a_1}\cdots x_\mu^{a_\mu}$ in the power series expression of $f_e$ that has degree equal to $\Theta_e(I)$ and such that $a_i < p^e$ for all $i$.
Since $S$ is regular there exists an $S$-linear map $\Phi \colon F_*^e(S) \to S$ such that $x_1^{p^e - 1}\cdots x_\mu^{p^e - 1}\mapsto 1$. Using the special monomial, we obtain that 
 $1 \in \Phi (\m^{\mu (p^e - 1) - \Theta_e(I)} f_e)$,
and therefore $\m^{\mu (p^e - 1) - \Theta_e(I)} f_e \not\subseteq I_e(S) = \m^{[p^e]}$. Since $f_e \in I^{[p^e]}:_S I$ we have that $\m^{\mu (p^e - 1) - \Theta_e(I)} \not\subseteq \m^{[p^e]}:(I^{[p^e]}:_SI)$, and thus the image of $\m^{\mu (p^e - 1) - \Theta_e(I)}$ in $S/I$ is not contained in $I_e(R)$ by Remark \ref{rem: Fedder}. This shows that $\loewy_R(R/I_e(R)) \geq \mu (p^e - 1) - \Theta_e(I)$, and concludes the proof.
\end{proof}

\begin{remark}\label{rem: local theta}
It follows from Theorem~\ref{thm: main formula} that if $R$ is an F-finite F-pure local ring, then $\Theta_e (I)$ is independent 
of the presentation $R = S/I$ as long as $(S,\m)$ is an F-finite regular local ring and $I \subseteq \m^2$, which we may always assume. 
Also, if we write $\widehat{S} = \kk\ps{x_1, \ldots, x_d}$ and 
set $T = \mathbb{L}\ps{x_1, \ldots, x_d}$, then $\Theta_e (I)=\Theta_e (I \widehat{S} ) = \Theta_e (IT)$. This is due to faithful flatness of the extensions $S \to \widehat{S} \to T$ and the fact that the maximal ideal of $S$ extends to the maximal ideals of $\widehat{S}$ and $T$. In this way, $I^{[p^e]}T :_T IT = (I^{[p^e]} :_{\widehat{S}} I)T \subseteq \m^nT +  \m^{[p^e]}T$
holds if and only if the original containment held. Summarizing, the value of $\Theta_e (I)$ does not change under completion and 
field extensions.
\end{remark}

Using the convergence estimates for F-pure thresholds, we provide a uniform convergence result for the functions $\Theta_e$. In what follows, if $I \subseteq S$ is an ideal we let $V(I)$ denote the set $\{\p \in \Spec(S) \mid I \subseteq \p\}$. By a slight abuse of notation, we often identify $V(I)$ with $\Spec(S/I)$.

\begin{corollary} \label{uniform convergence fibers} 
If $S$ is an F-finite regular ring, and $R=S/I$ is F-pure, then for all $e >0$ and all $\q \in V(I)$ one has
\[
\bigg| \frac{\Theta_e(I_\q)}{p^e} - \dim (S_\q) + \fpt (R_\q) \bigg| < \frac{\dim (S)}{p^e}.
\]
\end{corollary}
\begin{proof}
By plugging the formula of Theorem~\ref{thm: main formula} into Proposition~\ref{prop: fpt convergence}
we obtain that 
\[
0 \leq \fpt (R_\q) - \frac{(p^e - 1) \dim (S_\q) - \Theta_e (I_\q)}{p^e} \leq \frac{\dim (S_\q)}{p^e},
\]
so it follows that 
\[
0 \leq \frac{\Theta_e (I_\q)}{p^e} - (\dim (S_\q) - \fpt (R_\q)) \leq \frac{\dim (S_\q)}{p^e}
\leq \frac{\dim (S)}{p^e}. \qedhere
\]
\end{proof}

The following corollary was suggested to us by K. E. Smith. 

\begin{corollary}\label{cor: thanks Karen}
Let 
$R = \kk[x_1, \ldots, x_m]/I$ and $S = \kk[y_1, \ldots, y_n]/J$.
Then 
\[
\dfpt ((R \otimes_\kk S)_{(x_1, \ldots, x_m, y_1, \ldots, y_n)}) = \dfpt (R_{(x_1, \ldots, x_m)}) + \dfpt (S_{(y_1, \ldots, y_n)}).
\]
\end{corollary}
\begin{proof}
Let $A = \kk[x_1, \ldots, x_m], B = \kk[y_1, \ldots, y_n], C = \kk[x_1, \ldots, x_m, y_1, \ldots, y_n]$. Then $R \otimes_\kk S = C/\a$ where the ideal $\a$ is generated by the images of $I$ and $J$ in $C$. We now use Theorem~\ref{thm: main formula}. One can readily check that 
\[
\a^{[p^e]} :_C \a + (x_1, \ldots, x_m, y_1, \ldots, y_n)^{[p^e]} 
= (I^{[p^e]} :_A I)(J^{[p^e]} :_B J)C + (x_1, \ldots, x_m, y_1, \ldots, y_n)^{[p^e]}. 
\]
It follows from this formula that 
$\Theta_e (\a) = \Theta_e (I) + \Theta_e (J)$. For instance, one has
\[
\a^{[p^e]} :_C \a  \subseteq 
(x_1, \ldots, x_m, y_1, \ldots, y_n)^{\Theta_e (I) + \Theta_e (J)}
+ (x_1^{p^e}, \ldots, x_m^{p^e}, y_1^{p^e}, \ldots, y_n^{p^e})
\]
and 
\[
\a^{[p^e]} :_C \a \not\subseteq 
(x_1, \ldots, x_m, y_1, \ldots, y_n)^{\Theta_e (I) + \Theta_e (J) + 2}
+ (x_1^{p^e}, \ldots, x_m^{p^e}, y_1^{p^e}, \ldots, y_n^{p^e}). \qedhere
\]
\end{proof}

We can now give an expression for $\Theta_e$ in the localization in terms of differential operators and differential powers.

\begin{proposition}\label{prop: differential theta}
Let $(S,\m)$ be an F-finite regular local ring, and $I\subset S$ be an ideal such that $R=S/I$ is F-pure. For any prime ideal $I\subseteq \p\subset S$ we have
\begin{align*}
\Theta_e(IS_\p) &= \max \{n \mid I^{[p^e]} :_S I \subseteq \p^{\{n,p^e\}}\}\\
&= 
\max \{n \mid \delta (I^{[p^e]} :_S I) \subseteq \p \text{ for all }
\delta \in D^{(n, e)}_S \}.
\end{align*}
\end{proposition}
\begin{proof}
By definition, 
$
\Theta_e (IS_\p) = \max \{n \mid I^{[p^e]}S_\p :_{S_\p} IS_\p \subseteq (\p S_\p)^{\{n,p^e\}}\}$. 
Since $I^{[p^e]}S_\p :_{S_\p} IS_\p = (I^{[p^e]} :_S I)S_\p$,  and since $\p^{\{n,p^e\}}$ is $\p$-primary by \ref{prop:Analogue of ZN}, we have that 
\[
I^{[p^e]}S_\p :_{S_\p} IS_\p \subseteq (\p S_\p)^{\{n,p^e\}} \iff (I^{[p^e]}:_S I)S_\p \subseteq (\p^{\{n,p^e\}})S_\p \iff I^{[p^e]}:_S I \subseteq \p^{\{n,p^e\}}.
\]
The second equality follows directly from the definition of $\p^{\{n,p^e\}}$.
\end{proof}

\section{Defect of the F-pure thresholds for rings and schemes}\label{sect: global}

\subsection{Semi-continuity of the defect of the F-pure threshold}
We start this subsection by recalling the following definitions.

\begin{definition}
A real-valued function $f$ on a topological space $X$ is (strongly) upper semi-continuous if 
for any $a \in \mathbb{R}$ the set $\{x \in X \mid f(x) < a\}$ (respectively, $\{x \in X \mid f(x) \leq a\}$) is open. 
\end{definition}

A strongly upper semi-continuous function is upper semi-continuous because 
\[
\{x \in X \mid f(x) < a\} = \bigcup_{\varepsilon > 0} \{x \in X \mid f(x) \leq a - \varepsilon\}.
\]

\begin{lemma}\label{lemma semicont}
Let $S$ be an F-finite regular ring, and $I \subseteq R$ be an ideal such that $R=S/I$ is F-pure. Then the function $\Spec(R) \to \R$ defined as $\p \mapsto \dim (S_\p) - \fpt (R_\p)$ is upper semi-continuous. 
\end{lemma}
\begin{proof}
By Corollary~\ref{uniform convergence fibers} the functions $\p \mapsto \Theta_e (IS_\p)/p^e$ converge uniformly to their limit $\dim (S_\p) - \fpt (R_\p)$.
We demonstrate that the individual $\Theta_e$ are upper semi-continuous, 
which imply that their uniform limit is upper semi-continuous. 

Suppose that $\Theta_e (IS_{\p}) < \lambda$. By Proposition~\ref{prop: differential theta} there exists $\delta \in D^{(\lceil \lambda \rceil, p^e)}_R$
and $f \in I^{[p^e]} :_S I$ such that $\delta (f) \notin \p$. Let $U = \{\q \in V(I) \mid \delta(f) \notin \q\}$, then $\Theta_e(IS_\q) \leq \lceil \lambda \rceil - 1 \leq \lambda$
for all $\q \in U$. It follows that $\p \mapsto \Theta_e (IS_\p)$ is upper semi-continuous. 
\end{proof}

\begin{theorem} \label{thm semicont}
Let $R$ be an F-finite F-pure ring of characteristic $p > 0$. Then the  functions on $\Spec(R) \to \mathbb {R}$ defined by
\[
\begin{array}{ccc}
         \p & \mapsto & \mfpt(R_\p) \coloneqq \edim(R_\p) - \fpt (R_\p) \\
         \p & \mapsto  & \dfpt (R_\p)
\end{array}
\]
are upper semi-continuous. 
\end{theorem}
\begin{proof}
Fix a presentation $R= S/I$ where $S$ is a regular F-finite ring \cite{Gabber}. Because  $\dim(S_\p) = \dim(R_\p) + \height(IS_\p)$, as $S_\p$ is a regular local ring, Lemma \ref{lemma semicont} shows that the function $\p \mapsto \dfpt (R_\p) + \height (IS_\p)$ is upper semi-continuous. It now suffices to prove that $\p \mapsto \height (IS_\p)$ is lower semi-continuous. Namely, suppose that $\p$ is such that $\height (IS_\p) \geq h$. Let $Q_1,\ldots,Q_t$ be the minimal primes of $I$ satisfying $\height(Q_i)<h$, and let $U'=V(I) \smallsetminus V(Q_1\cap \ldots \cap Q_t)$. We have that $\p \in U'$ and 
$\height(IS_\q) \geq h$ for all $\q \in U'$.

We prove the second assertion in a similar way, by showing that $\phi \colon \p \mapsto \dim (S_\p) - \edim (R_\p)$ is lower semi-continuous. We use Nagata's criterion for openness. 
First, given that $\p \subseteq \q$ we want to show that $\phi (\p) \geq \phi (\q)$. By localizing at $\q$, we assume that $S$ and $R$ are local and $\q$ is the maximal ideal. By definition, we may choose a parameter ideal $J$ in $I$ such that $\dim (S/J) = \edim (R)$ and $S/J$ is still regular. Then 
\[
\phi(\q) = \height (J) = \height (J_\p) = \dim (S_\p) - \dim (S_\p/J S_\p) 
\leq \dim (S_\p) - \edim (R_\p) = \phi(\p). 
\]
Second, given a prime ideal $\p$ we need to assure that $\phi$ is constant in an open set of $V(\p)$. Similarly to above, there is a parameter ideal $J \subseteq IS_\p$ such that 
$\dim (S_\p/J) = \edim (R_\p)$ and $S_\p/J$ is still regular.
We may invert an element $a \notin \p$ so that $J$ is a parameter ideal of $IS_a$
and, because the regular locus is open, we may invert an element $b \notin \p$
so that $(S/J)_{ab}$ is regular. Then for any prime $ab \notin \q$ 
we have that 
\[
\phi (\q) =
\dim (S_\q) - \edim (R_\q)
= \height (J_\q) + \dim (S_\q/JS_\q) - \edim (R_\q) \geq \height (J) =  \phi (\p).
\]
Thus $\phi (\q) = \phi (\p)$ if $\p \subseteq \q$ by the first inequality.
\end{proof}

\begin{remark}\label{rmk: semicontinuity}
Semi-continuity still holds if we define $\dfpt (R) = \infty$ when $R$ is not F-pure, consistent with the fact that we have defined $\fpt(R)=-\infty$ in this case. Since the F-pure locus is open, the theorem still provides that the set 
\[
\{\p \mid \dfpt (R_\p) < a\}
\]
is open for every $a \in \mathbb{R} \cup \{\infty\}$.
\end{remark}

Any upper semi-continuous function on the spectrum of a Noetherian ring 
satisfies the ascending chain condition. However, due to a  result of Sato \cite{SatoACC}, we also have the decreasing chain condition in the $\mathbb Q$-Gorenstein case. Recall that, if $R$ is a normal ring, and $I \subseteq R$ is an ideal isomorphic to a canonical module of $R$, then the anticanonical cover of $R$ is $A= \bigoplus_{n \geq 0} J^{(n)}$, where $J$ is the inverse of $I$ in the divisor class group of $R$.

\begin{lemma}\label{LemmaAnticanonical}
Let $R$ be an F-finite F-pure normal ring. 
Suppose that $R$ has a finitely generated anticanonical cover $A$, and let $\a\subseteq R$ be an ideal.
Then $\fpt(\a,R)=\fpt(\a A,A)$.
\end{lemma}
\begin{proof}
Let $\cB_e=\{t\in\NN\; |\; \Hom_R(F^e_*(R),R)\cdot F^e_*(I^t)=R\}$
and
 $$\cD_e=\{t\in\NN\; |\; \Hom_A(F^e_*(A),A)\cdot F^e_*(I^t A)=A\}.$$
From Definition \ref{DefFptGlobalIdeal}, it suffices to show that $\cB_e=\cD_e$ for every $e\geq 0$. We show this by double containment. Let $t\in\cB_e$. Then there exist $f_1, \ldots, f_n \in I^t$ and $\phi_1, \ldots, \phi_n \in \Hom_R(F^e_*(R),R)$ such that $\sum_{i = 0}^n \phi_i(F^e_*(f_i))=1$. For every $e \geq 0$ 
every map $\phi\in\Hom_R(F^e_*(R),R)$ has an extension
$\widetilde{\phi}\in \Hom_R(F^e_*(A),A)$ \cite[Lemmas 3.1 \& 3.2]{CEMS}.
Let $\widetilde{\phi_i}\in Hom_R(F^e_*(A),A)$ be extensions of $\phi_i$. Then,
$\sum_{i = 1}^n \widetilde{\phi}_i(F^e_*(f_i))=1$, and so,  $\Hom_A(F^e_*(A),A)\cdot F^e_*(I^t A)=A$.

Let $t\in\cD_e$. By definition we can find $g_1, \ldots, g_n \in I^tA$ and $\varphi_1, \ldots, \varphi_n \in \Hom_A(F^e_*(A),A)$ be such that $\sum_{i = 0}^n \varphi_i(F^e_*(g_i))=1$. We may assume that $g_i \in I^t$. Namely, if $g_i = \sum_{j = 1}^n f_{ij} a_{ij}$ where all $f_{ij} \in I^t$ and $a_{ij} \in A$, then the maps $\varphi_{ij}$ obtained by composing $\varphi_i$ with the multiplication by $F^e_*(a_{ij})$ satisfy
$\varphi_i (F^e_*(g_i)) = \sum_{j = 1}^n \varphi_{ij} (F^e_*(f_{ij}))$.

Now, let $\iota\colon F^e_*(R) \to F^e_*(A)$ be the inclusion, and $\rho\colon A\to R$ be the projection onto the degree zero component of $A$. We note that $\gamma_i=\rho\circ \varphi_i\circ \iota\in \Hom_R(F^e_*(R),R)$. Then $\sum_{i = 1}^n \gamma_i(F^e_*(g_i))=1$, and thus  $\Hom_R(F^e_*(R),R)\cdot F^e_*(I^t)=R$.
\end{proof}

\begin{lemma}\label{LemmaACC}
Let $R$ be an F-finite F-pure normal ring.
Suppose that $R$ has a finitely generated anticanonical cover $A$.
Then the set
$$
\{\fpt(R_\q) \mid \q\in\Spec(R) \}
$$
satisfies the ascending chain condition.
\end{lemma}
\begin{proof}
Because $A$ is F-finite there is an F-finite regular ring $S$ that surjects onto $A$. Since the embedding dimension of $A_\q$ is bounded by $\dim(S)$ and $A$ is a quasi-Gorenstein\footnote{In the published version of this article, it is claimed that $A$ is Gorenstein. However, in our assumptions one can only say that $A$ is quasi-Gorenstein; this correction does not affect the rest of the proof.} ring \cite[Theorem 2.2  and (3.1)]{WatanabeCover}, the set 
\[
\mathcal{F}_A \coloneqq \{\fpt(I A_Q,A_Q) \mid Q \in\Spec(A) \text{ and } \sqrt{I} = Q \}
\]
satisfies the ascending chain condition \cite[Theorem 4.7]{SatoACC}.
By Lemma \ref{LemmaAnticanonical} we have that 
$$
\fpt(R_\q)=\fpt(\q R_\q,R_\q)=\fpt(\q A_\q,A_\q), 
$$
therefore $\{\fpt(R_\q) \mid \q\in\Spec(R) \} \subseteq \mathcal{F}_A$ and 
the assertion follows. 
\end{proof}

We use the result to establish a stronger form of semi-continuity. 

\begin{theorem} \label{thm constructible}
Let $R$ be an F-finite F-pure Cohen-Macaulay normal ring. 
Suppose that $R$ has a finitely generated anticanonical cover $A$ (for example, $R$ is $\mathbb Q$-Gorenstein).
Then the set $\{\fpt(R_\q) \mid \q \in \Spec (R)\}$ is finite, and the functions 
\[
\q \mapsto \dfpt (R_\q) \ \text { and } \ \q \mapsto \mfpt(R_\q) \coloneqq \edim (R_\q) - \fpt (R_\q),
\]
are strongly upper semi-continuous. Moreover, they define finite stratifications of $\Spec(R)$ with constructible strata. 
\end{theorem}
\begin{proof}
By Theorem~\ref{thm semicont} the sets  
$\{\dfpt(R_\q) \mid \q \in \Spec (R)\}$ and $\{\mfpt(R_\q) \mid \q \in \Spec (R)\}$ satisfy the ascending chain condition (\emph{e.g.}, see \cite[Theorem~2.3]{SmirnovTucker}). By Lemma~\ref{LemmaACC} they also satisfy the descending chain condition.
Hence the sets are finite. 

For the second assertion we prove that every stratum of a finitely valued upper semi-continuous function is locally closed. Namely, for any given $a$ we may choose $\varepsilon > 0$ such that 
$\{\q \in \Spec (R) \mid \dfpt(R_\q) \leq a\} = \{\q \in \Spec (R) \mid \dfpt(R_\q) < a + \varepsilon\}$. It follows that
\begin{align*}
\{\q \in \Spec (R) \mid \dfpt(R_\q) = a\} 
&= 
\{
\q \in \Spec (R) \mid \dfpt(R_\q) \geq a \text{ and }
\dfpt(R_\q) < a + \varepsilon\}
\\
&= 
\{
\q \mid \dfpt(R_\q) \geq a\} \cap 
\{
\q  \mid 
\dfpt(R_\q) < a + \varepsilon\}. \qedhere
\end{align*}
\end{proof}

\subsection{Global defect of the F-pure threshold}

Our next goal is to define an invariant for an F-finite F-pure ring, not necessarily local, which encodes the local behavior of the defect of the F-pure threshold. This type of process has been carried out before for other numerical invariants of rings of characteristic $p>0$ \cite{DSPY,DSPY1,DSPY2}. 

\begin{definition}\label{DefTheta}
Let $S$ be a regular F-finite ring and $I\subsetneq S$ be an ideal such that $R=S/I$ is an F-pure ring. We define 
\[
\Theta_e(I) \coloneqq \max \left\{n \ \bigg| \ \left( \delta (\frpe{I} :_S I) \mid \delta \in D^{(n, p^e)}_S \right) \text{ is a proper ideal} \right\},
\]
and
\[
\Theta(I) \coloneqq \lim_{e \to \infty} \frac{\Theta_e(I) }{p^e}.
\]
\end{definition}

\begin{remark} \label{remark Fedder Theta} 
In light of Proposition~\ref{prop:Analogue of ZN} and Proposition~\ref{prop: differential theta}, one could also define $\Theta_e(I)$ as follows:
\begin{align*}
\Theta_e(I) & = \max\left\{n \ \bigg| \ \frpe{I}:_S I \not\subseteq (\p^n+\frpe{\p})S_\p \cap S \text{ for all } \p \in \Spec(R) \right\} \\
& = \max\left\{n \ \bigg| \ \frpe{I}:_S I \not\subseteq \m^n + \m^{[p^e]} \text{ for all } \m \in \Max\Spec(R) \right\}.
\end{align*}
This alternative point of view, which more closely resembles Fedder’s criterion, can be useful to keep in mind as it could give a more direct approach to some problems. 
However, for the purposes of this article, we have chosen to maintain a unified approach based on differential operators, as in our view this often leads to more intrinsic and global statements. 
\end{remark}

\begin{proposition}\label{prop: global theta}
Let $S$ be an F-finite regular ring and $R = S/I$ be F-pure. Then
the sequence $\Theta_e (I)/p^e$ converges and its limit $\Theta(I)$ satisfies
\[
\Theta (I) = \max \{\dim (S_\p) - \fpt (R_\p) \mid \p \in \Spec (R)\}.
\]
\end{proposition}
\begin{proof}
Let $\eta = \max \{\dim (S_\p) - \fpt (R_\p) \mid \p \in \Spec (R)\}$, which exists by Lemma~\ref{lemma semicont}.
By Proposition~\ref{prop: differential theta} we have that 
$\Theta_e (I) = \max \{\Theta_e (IS_\q) \mid \q \in V(I) \}$, so that 
for each $e$ we can find a prime ideal $\q_e$ such that $\Theta_e (IS_{\q_e}) = \Theta_e (I)$.
We want to show that $\eta = \lim_{e \to \infty} \Theta_e (IS_{\q_e})/p^e$. 

By Corollary \ref{uniform convergence fibers}, given any $\varepsilon>0$ there exists $e_0 \geq 0$ such that, for all $e\geq e_0$ and all primes $\q \in V(I)$, one has 
\[
\left|\frac{\Theta_e (IS_\q)}{p^e} - \dim(S_\q) + \fpt(R_\q) \right| < \frac{\varepsilon}{2}.
\]
Let $q_{\max}$ be any prime ideal such that $\dim(S_{\q_{\max}}) - \fpt(R_{\q_{\max}}) = \eta$. Since $\Theta_e (IS_{\q_{\max}}) \leq \Theta_e (IS_{\q_e})$ by our choice of $\q_e$, 
for any $e \geq e_0$ we obtain inequalities
\[
\eta \geq \dim(S_{\q_e}) - \fpt(R_{\q_e}) \geq 
\frac{\Theta_e (IS_{\q_e})}{p^e} - \frac{\varepsilon}{2} \geq 
\frac{\Theta_e (IS_{\q_{\max}})}{p^e} - \frac{\varepsilon}{2} 
\geq \eta - \varepsilon
\]
and the assertion follows.
\end{proof}

\begin{theorem} \label{thm: global (new)}
Let $R$ be an F-finite F-pure ring which is either 
\begin{enumerate} 
\item a domain, or 
\item biequidimensional.
\end{enumerate} 
Let $S$ be an F-finite regular ring mapping onto $R$, which we assume being coequidimensional in the second case (such representation always exists by Proposition~\ref{prop: equidimensional Gabber}). Write $R \cong S/I$ for some ideal $I \subseteq S$. 
We define the global defect of the F-pure threshold by
\[
\dfpt(R) \coloneqq \Theta (I) - \height(I).
\]
Then $\dfpt(R) = \max \{\dfpt (R_\p) \mid \p \in \Spec (R)\}$. In particular, it
is independent of the presentation (coequidimensional presentation in the second case) and, therefore, is well-defined. 
\end{theorem}
\begin{proof}
Proposition~\ref{prop: global theta} shows that 
\[
\dfpt (R) + \height (I) = \max \{\dfpt (R_\p) + \height (IS_\p)\mid \p \in \Spec (R)\}, 
\]
so it suffices to verify that 
$\height (IS_\p) = \height (I)$ for all $\p \in V(I)$. If $I$ is a prime ideal, this is trivial. 
Otherwise, let $Q$ be a minimal prime ideal of $I$. 
For any maximal ideal $\m$ of $V(I)$ we then have 
\[
\height (Q) = \dim (S_\m) -  \dim (S_\m/QS_\m) 
= \dim (S) - \dim (S_\m/QS_\m). 
\]
However, $\dim (S_\m/QS_\m)  = \dim (R_\m/QR_\m) = \dim(R)$ as $R$ is biequidimensional, so $\height(Q) = \dim(S) - \dim(R)$ is independent of $Q$, and the assertion follows. 
\end{proof}

The following example shows that the assumptions of Theorem \ref{thm: global (new)} are needed on $R$ and its regular presentation in order for local and global invariants to patch up.

\begin{example} \label{bad ex1} Let $S=\kk\ps{x,t} \times \kk\ps{y,z}$ and $I = (0) \times (yz)$. Note that $S$ is coequidimensional, but $R=S/I$ is not equidimensional. We have that $\height(I)=0$ and $\Theta_e(I) = \max\{\Theta_e(IS_\p) \mid \p \in V(I)\} = 2(p^e-1)$. It follows that $\lim\limits_{e \to \infty} \Theta_e(I)/p^e - \height(I) = 2$, while $\max\{\dim(R_\p) - \fpt(R_\p) \mid \p\in \Spec(R)\} = 1$, attained at $\p = \kk\ps{x,t} \times (y,z)$.
\end{example}
\begin{example} \label{bad ex2} Let $S=\kk\ps{x} \times \kk\ps{y,z}$ and $I = (0) \times (yz)$. Note that $R=S/I$ is biequidimensional, but its regular presentation $S$ is not coequidimensional. As in Example \ref{bad ex1} we have that $\height(I)=0$ and $\Theta_e(I) = \max\{\Theta_e(IS_\p) \mid \p \in V(I)\} = 2(p^e-1)$, so that $\lim\limits_{e \to \infty} \Theta_e(I)/p^e - \height(I) = 2$. However, $\max\{\dim(R_\p) - \fpt(R_\p) \mid \p\in \Spec(R)\} = 1$, attained at the prime ideal $\p = \kk\ps{x} \times (y,z)$.
\end{example}

\begin{proposition} 
Let $R$ be an F-finite F-pure ring which is either a domain or biequidimensional. Then the following are equivalent
\begin{enumerate}
\item $\dfpt(R)  <1$,
\item $\dfpt(R) =0$,
\item $R$ is regular.
\end{enumerate}
\end{proposition}
\begin{proof}
The claim can be reduced to local rings by Proposition~\ref{prop: global theta}, and then the result follows from Proposition~\ref{prop TW regular}.
\end{proof}

Now we show that  the maximum value of the defect of the F-pure threshold
of a positively graded algebra over a field
is achieved at the homogeneous maximal ideal. This result is analogous to those obtained for Frobenius Betti numbers \cite{DSPY2} and F-splitting ratio \cite{DSPY1}.

\begin{proposition} Let $R$ be an F-finite and F-pure positively graded $\kk$-algebra which is either a domain or biequidimensional. Let $\m$ denote the unique homogeneous maximal ideal of $R$, that is, the ideal generated by elements of $R$ of positive degrees. Then $\dfpt(R) = \dfpt(R_\m)$.
\end{proposition}
\begin{proof}
We may write $R$ as a quotient of $S=\kk[x_1,\ldots,x_n]$,  $\deg(x_i)=d_i>0$, by a homogeneous ideal $I \subseteq S$.  We now show that $D^{(n,p^e)}_S$ is a $\ZZ$-graded $S$-module for every $n,e \geq 0$. Since $S$ is F-finite, 
$$
\frac{S\otimes_\ZZ S}{\Delta^{n+1}_{S|\ZZ}+\Delta^{[p^e]}_{S|\ZZ}} \cong \frac{S \otimes_{S^{p^e}} S}{\Delta_{S|S^{p^e}}^{n+1}}
$$
and
$$\Hom_R\left( \frac{R\otimes_\ZZ R}{\Delta^{n+1}_{R|\ZZ}+\Delta^{[p^e]}_{R|\ZZ}}, R\right) \cong D^{(n,p^e)}_{R}
$$
by Remark \ref{PPFinite}.
Since $\frac{S\otimes_\ZZ S}{\Delta^{n+1}_{S|\ZZ}+\Delta^{[p^e]}_{S|\ZZ}} $ is a graded module, we have that $D^{(n,p^e)}_{R}$ is also graded.

Since $I^{[p^e]}:I$ is a homogeneous ideal and $D^{(n,p^e)}_S$ is graded, 
the ideal $\left(\delta(I^{[p^e]}:I) \mid \delta \in D^{(n,p^e)}_S \right)$ is homogeneous, and therefore it is proper if and only if its localization at $\m$ is proper. On the other hand, $D^n_{S} \subseteq D^{(e)}_S$ for some $e$, thus $D^n_{S}$ is a finitely generated $S$-module, so we conclude that $(D^n_S)_\m \cong D^n_{S_\m}$ with an argument similar to the proof of Lemma~\ref{LemmaDiffPowersLocalization}. 
Therefore,
$
\left(\delta(I^{[p^e]}:I) \mid \delta \in D^{(n,p^e)}_S \right)_\m \cong \left(\delta'(I^{[p^e]}S_\m:IS_\m) \mid \delta' \in D^{(n,p^e)}_{S_\m}\right).
$
It follows that $\Theta(IS_\m) = \Theta(I)$, and, since $\height(I) = \height(IS_\m)$ because $I$ is homogeneous, the proof is complete.
\end{proof}

We end the section by extending the definition of $\dfpt$ to schemes. Let $X$ be a Noetherian $\FF_p$-scheme. We say that $X$ is F-finite if there exists a finite open affine cover $\{U_i = \Spec(R_i)\}_{i=1}^t$ such that $R_i$ is F-finite. An F-finite scheme is F-pure if each $R_i$ is F-pure. In particular, an F-pure scheme is reduced.
\begin{definition}\label{Def dpt(X)}
If $X$ is a Noetherian F-finite F-pure scheme, we define 
\[
\dfpt(X) = \max \{\dfpt(\OO_{X,x}) \mid x \in X\}.
\]
\end{definition}
Note that, with the notation introduced above, we have 
$\dfpt(X) = \max\{d_i \mid i=1,\ldots,t\}$ where $d_i = \sup\{\dfpt((R_i)_\p) \mid \p \in \Spec(R_i)\} = \max\{\dfpt((R_i)_\p) \mid \p \in \Spec(R_i)\}$ by Theorem \ref{thm semicont}. 
If the rings $R_i$ are either domains or biequidimensional, then $d_i = \dfpt(R_i)$ as shown in Theorem \ref{thm: global (new)}.

\subsection{Global Fedder's criteria}
Let $R$ be an F-finite local ring. There exists an F-finite regular local ring $S$ mapping onto $R$; write $R=S/I$ \cite{Gabber}. When $(R,\m)$ is local, Fedder's criterion as already recalled in Theorem \ref{FedderCriterion} allows to characterize when $R$ is F-pure with a very explicit calculation. Such a criterion, however, heavily depends on the fact that $R$ is local. Even when $R$ is not local, since an F-finite ring $R$ is F-pure if and only if $R_\m$ is F-pure for every maximal ideal $\m$, one could in theory apply Fedder's criterion to every localization to test whether $R$ is F-pure or not. In practice, this is often not feasible. Instead, the description of $\Theta_e$ in terms of differential operators, gives a ``global'' condition that characterizes F-purity in a way analogous to Fedder's, but which holds for F-finite rings which are not necessarily local.

\begin{proposition}\label{prop: global Fedder}
Let $S$ be a F-finite regular ring, $I\subseteq S$ be an ideal, and $R=S/I$.
Let $n=\dim(S)$.
Then $R$ is F-pure if and only if $\left(D^{(n(p^e-1),p^e)}_S \left( I^{[p^e]}:I\right)\right) = S$ for some (equivalently, for every) integer $e>0$.
\end{proposition}
\begin{proof}
We fix $e>0$.
We have that
\begin{align*}
R\hbox{ is F-pure }& \Longleftrightarrow R_\m \hbox{ is F-pure for every }\m\in\MaxSpec(S)\\
& \Longleftrightarrow \widehat{R_\m} \hbox{ is F-pure for every }\m\in\MaxSpec(S)\\
& \Longleftrightarrow I^{[p^e]}\widehat{S_\m}:I \widehat{S_\m}\not\subseteq \m^{[p^e]}\widehat{S_\m} \hbox{ for every }\m\in\MaxSpec(S)\\
& \Longleftrightarrow \left(D^{(n(p^e-1),p^e)}_{\widehat{S_m}}\left( I^{[p^e]}\widehat{S_\m}:I \widehat{S_\m}\right)\right) \not\subseteq \m \widehat{S_\m} \hbox{ for every }\m\in\MaxSpec(S)\\
& \Longleftrightarrow \left(D^{(n(p^e-1),p^e)}_{S_m}\left( I^{[p^e]}S_\m:I S_\m\right)\right) =S_\m \hbox{ for every }\m\in\MaxSpec(S)\\
& \Longleftrightarrow \left(D^{(n(p^e-1),p^e)}_{S} \left( I^{[p^e]}:I\right)\right) =S.
\end{align*}
The proof for every $e>0$ is analogous, and it is left to the reader.
\end{proof}

We recall the definition of geometrically F-pure  algebras introduced by
Schwede and Zhang \cite[Definition 2.1]{SchwedeZhang}.

\begin{definition} \label{def: geometrically F-pure} Let $\kk$ be an F-finite field. We say that a finite type $\kk$-algebra $R$ is geometrically F-pure if $R_{\kk'}=R \otimes_\kk \kk'$ is F-pure for every finite extension $\kk \subseteq \kk'$.
\end{definition}

One can check that $R$ is geometrically F-pure if and only if  $R_{\overline{\kk}}$ is F-pure, where $\overline{\kk}$ denotes an algebraic closure of $\kk$. Recall that, when $\kk$ is perfect, one has $D_S^{(n,p^e)} = D_{S|\kk}^{(n,p^e)}$ for all $n,e$ (see Remark \ref{rem: perfect K}). However, in general, the containment $D_{S|\kk}^{(n,p^e)} \subseteq D_S^{(n,p^e)}$ can be strict; for instance, see the forthcoming Example~\ref{ex: Fedder}. This difference allows  a suitable modification of Proposition~\ref{prop: global Fedder} to test for geometric F-purity.

\begin{proposition}\label{prop: geo Fedder}
Let $\kk$ be an F-finite field, $S=\kk[x_1,\ldots,x_n]$, $I\subseteq S$ be an ideal, and $R=S/I$.
Then $R$ is geometrically F-pure if and only if $\left(D^{(n(p^e-1),p^e)}_{S|\kk}\left( I^{[p^e]}:I \right)\right)= S$ for some (equivalently, for every) $e>0$.
\end{proposition}
\begin{proof}
Fix $e>0$, and let $S_{\overline{\kk}}$ denote $S\otimes_\kk \overline{\kk}$.
We have that
\begin{align*}
R\hbox{ is geometrically F-pure }& \Longleftrightarrow R\otimes_\kk \overline{\kk} \hbox{ is F-pure}\\
& \Longleftrightarrow \left(D^{(n(p^e-1),p^e)}_{S_{\overline{\kk}}}\left( I^{[p^e]}S_{\overline{\kk}}:IS_{\overline{\kk}} \right) \right)= S_{\overline{\kk}}\\
& \Longleftrightarrow \left(D^{(n(p^e-1),p^e)}_{S_{\overline{\kk}}|\overline{\kk}} \left( I^{[p^e]}S_{\overline{\kk}}:IS_{\overline{\kk}} \right)\right)= S_{\overline{\kk}} \hbox{ because }\overline{\kk}\hbox{ is perfect}\\
& \Longleftrightarrow \left(\left(D^{(n(p^e-1),p^e)}_{S|\kk}\otimes_\kk \overline{\kk}\right) \left( I^{[p^e]}S_{\overline{\kk}}:IS_{\overline{\kk}}\right)\right) = S_{\overline{\kk}}\\
& \Longleftrightarrow \left(D^{(n(p^e-1),p^e)}_{S|\kk} \left( I^{[p^e]}:I\right)\right) = S.\\
\end{align*}
The proof of the equivalence for every $e>0$ is analogous.
\end{proof}

\begin{remark} We note, as a consequence of Propositions \ref{prop: global Fedder} and \ref{prop: geo Fedder}, that an F-finite ring $R=S/I$ with $S=\kk[x_1.\ldots,x_n]$ is geometrically F-pure if and only if $R_{\kk'}$ is F-pure for every extension $\kk \subseteq \kk'$ such that $\kk'$ is F-finite. In fact, every finite extension of $\kk$ is F-finite, since $\kk$ itself is F-finite. Conversely, assume that $R_{\overline{\kk}}$ is F-pure and $\kk'$ is an F-finite extension of $\kk$. Then $\left(D^{(n(p-1),p)}_{S|\kk}\left( I^{[p^e]}:I \right)\right)= S$ by Proposition~\ref{prop: geo Fedder}, and tensoring with $\kk'$ we get $\left(D^{(n(p-1),p)}_{S_{\kk'}|\kk'}\left( I^{[p^e]}S_{\kk'}:IS_{\kk'} \right)\right)= S_{\kk'}$. Since $D^{n(p-1)}_{S_{\kk'}|\kk'} \subseteq D^{n(p-1)}_{S_{\kk'}}$, we must also have 
\[
\left(D^{(n(p-1),p)}_{S_{\kk'}}\left( I^{[p^e]}S_{\kk'}:IS_{\kk'} \right)\right)= S_{\kk'}.
\]
It follows that $R_{\kk'}$ is F-pure by Proposition \ref{prop: global Fedder}.
\end{remark}

We now provide an example to show how these criteria work, and how to two notions of F-purity and geometrical F-purity can differ.

\begin{example} \label{ex: Fedder}
Let $\kk=\FF_2(u)$, the fraction field of the polynomial ring $\FF_2[u]$.
Let  $S=\kk[x]$, $f=x^2-u$, and $I=(f)$.
The ring $S/I$ is regular, and therefore F-pure. However, $S/I$ is not geometrically F-pure, as $(S/I)\otimes_\kk \kk(u^{1/2})$ is not even reduced.
We note that $\partial_u f^{p-1}=\partial_u f=-1$ and $\partial_u\in D^{(1,2)}_{S|\FF_2} $. Hence, 
$D^{(1,2)}_{S|\FF_2} \left( I^{[2]}:I\right) = S$.
However, $\partial_u\not\in D^{(1,2)}_{S|\kk}$. In fact, $D^{(1,2)}_{S|\kk}=S\cdot 1\oplus S\cdot\partial_x $, so
$ D^{(1,2)}_{S|\kk} \left( I^{[2]}:I\right) = D^{(1,2)}_{S|\kk}((f)) = (f)\neq S.$
\end{example}

\subsection{Bertini theorems for the defect of the F-pure threshold} 

We start by recalling the notions of equidimensionality and biequdimensionality for a scheme $X$. In case $X$ is affine they coincide with those given in Section \ref{sect: prelim}.

\begin{definition} Let $X$ be a finite dimensional Noetherian scheme. We say that $X$ is
\begin{itemize}
\item {\it equidimensional} if all irreducible components of $X$ have the same dimension,
\item {\it biequidimensional} if all maximal chains of irreducible closed subsets of $X$ have the same length.
\end{itemize}
\end{definition}

\begin{remark} \label{remark DattaSimpson} Let $X$ be a scheme of finite type over a field $\kk$. Then equidimensionality and biequidimensionality coincide \cite[Remark 2.5.2 (2)]{DattaSimpson}, and $X$ is (bi)equidimensional if and only if $X \times_\kk \kk'$ is (bi)equidimensional for any field extension $\kk \subseteq \kk'$ \cite[Proposition 2.5.3]{DattaSimpson}.
\end{remark}

\begin{theorem}[{\cite[Theorem 1]{CGM} \label{thm CGM}}]
Let $X$ be an equidimensional scheme of finite type over an algebraically closed field $\kk$ of characteristic $p>0$. Let $\phi\colon X \to \PP^n_{\kk}$ be a finite type morphism with separably generated (but not necessarily algebraic) residue field extensions (for example, an immersion). Suppose $X$ has a local property $\mathcal P$ satisfies the following axioms: 
\begin{itemize}
\item[(A1)]  Whenever $\phi \colon Y \to Z$ is a flat morphism of F-finite ring with regular fibers and $Z$ is $\mathcal P$, then $Y$ is $\mathcal P$.
\item[(A2)] Let $\phi \colon Y \to S$ be a morphism of finite type where $S$ is F-finite and integral with generic point $\eta$. If $Y_\eta$ is equidimensional and geometrically $\mathcal P$, then there exists an open neighborhood $U$ of $\eta$ in $S$ such that $Y_s$ is geometrically $\mathcal P$ for each $s \in U$.
\end{itemize}
Then there exists a nonempty open subset $U$ of $(\PP^n)^*$ such that $\phi^{-1}(H)$ has the property $\mathcal P$ for each hyperplane $H \in U$.
\end{theorem}
\begin{proof}
This theorem was proved in \cite{CGM} without 
the equidimensionality assumption on $X$ and with axioms that do not require F-finiteness. However, under our assumptions, $X$ and $\PP_{\kk}^n$ are F-finite and this implies that all schemes involved in the proof are F-finite,
see \cite[Discussion~4.2.1]{DattaSimpson} and especially the diagram (4.12) therein. 
Moreover, it was explained in \cite[Proposition~4.2.4]{DattaSimpson}
that the equidimensionality of $X$ implies equidimensional generic fiber $Y_\eta$ at the place where (A2) is invoked. 
\end{proof}

\begin{definition}
Given $\lambda \in \R_{\geq 0}$, we say that a Noetherian F-finite $\FF_p$-scheme $X$ satisfies $\mathcal P(\lambda)$ if it is F-pure and $\dfpt(X)<\lambda$, \emph{i.e.}, $\dfpt (\OO_{X,x}) < \lambda$ for all $x \in X$. 

If $X$ is a $\kk$-scheme of finite type, with $\kk$ an F-finite field, we say that it satisfies $\mathcal P(\lambda)$ geometrically if it is geometrically F-pure and $\dfpt(X \times_\kk \overline{\kk}) < \lambda$. 
\end{definition}

In order to show that Condition (A1) holds for $\mathcal P(\lambda)$, we need the following two results regarding flat extensions. We further study flat maps in Subsection \ref{sub: flat}.

\begin{proposition} \label{prop A1} Let $(A,\m_A) \to (R,\m_R)$ be a faithfully flat map of F-finite local rings such that $A$ is F-pure and $R/\m_A R$ is regular. Then
\[
\loewy_R(R/I_e(R)) = \loewy_A(A/I_e(A)) + \dim(R/\m_A R)(p^e-1).
\]
In particular, $\fpt(R) = \fpt(A) + \dim(R/\m_A R)$.
\end{proposition}
\begin{proof}
Our assumptions guarantee that $R$ is F-pure as well \cite[Proposition 4.8]{SchwedeZhang}. Let $t=\dim(R/\m_A R)$, and $x_1,\ldots,x_t \in R$ be elements whose images in $R/\m_A R$ form a regular system of parameters. Let $J=(x_1,\ldots,x_t)$. We have that $I_e(R) = I_e(A)R + J^{[p^e]}$ \cite[Claim 3.4]{CRST} \cite[Lemma 3.8]{AE}. Since $x_1,\ldots,x_t$ is a regular sequence modulo $\m_A R$, the quotient $R' \coloneqq R/J$ is a faithfully flat $A$-algebra (\cite[Section 21]{Matsumura}) such that  $\m_A R'= \m_{R'}$. By faithful flatness it then follows that 
\begin{align*}
\m_A^n \subseteq I_e(A) & \Longleftrightarrow I_e(A):_A\m_A^n = A \\
&  \Longleftrightarrow (I_e(A):_A\m_A^n)R' = R' \\
&  \Longleftrightarrow (I_e(A):_A \m_A^n)R' = R' \\
 &  \Longleftrightarrow (I_e(A)R':_{R'} \m_A^nR') = R' \\
  &  \Longleftrightarrow (I_e(A)R':_{R'} \m_{R'}^n) = R' \\
  & \Longleftrightarrow \m_{R'}^n \subseteq I_e(A)R' \\
  & \Longleftrightarrow \m_R^n \subseteq I_e(A)R+J.
\end{align*}
Hence $\loewy_A(A/I_e(A)) = \loewy_{R'}(R'/I_e(A)R') = \loewy_R(R/(I_e(A)R + J))$.

Assume that $\m_R^n \subseteq I_e(A)R + J$ or, equivalently, $\m_A^n \subseteq I_e(A)$. We have that
\[
\m_R^{t(p^e-1)+n} = (\m_AR+J)^{t(p^e-1)+n} \subseteq \m_A^nR + J^{t(p^e-1)+1} \subseteq I_e(A)R + J^{[p^e]}
\]
using the pigeonhole principle.

Conversely, assume that $r \in \m_R^n$ is such that $r \notin I_e(A)R + J$. Let $x=x_1\cdots x_t$, and note that $x^{p^e-1}r \in \m_R^{t(p^e-1)+n}$. For $x_1,\ldots,x_t$ is a regular sequence modulo $\m_A R$, we have that $x_1,\ldots,x_t$ is a regular sequence in $R/I_e(A)R$ \cite[Lemma~2.3]{PolstraSmirnov2}. As a consequence, we have that $x^{p^e-1}r\notin I_e(A)R + J^{[p^e]}$. These considerations show that $\m_R^n \subseteq I_e(A)R + J$ if and only if $\m_R^{t(p^e-1)+n} \subseteq I_e(A)R + J^{[p^e]}$ and, as a consequence, $\loewy_R(R/(I_e(A)R+J^{[p^e]})) = t(p^e-1) + \loewy_R(R/(I_e(A)R+J))$. Putting everything together, we conclude that
\begin{align*}
\loewy_R(R/I_e(R)) &  = \loewy_R(R/(I_e(A)R+J^{[p^e]})) \\
& = t(p^e-1) + \loewy_R(R/I_e(A)R+J)) = t(p^e-1) + \loewy_A(A/I_e(A)),
\end{align*}
as desired. The second equality follows now immediately after dividing by $p^e$ and taking limits.
\end{proof}

\begin{corollary} \label{coroll A1} Let $(A,\m_A) \to (R,\m_R)$ be a faithfully flat map of F-finite and F-pure local rings with $R/\m_A R$ regular. Then 
$\dfpt(R) = \dfpt(A)$.
\end{corollary}
\begin{proof}
We have that $\dim(R) = \dim(A) + \dim(R/\m_A R)$ by faithful flatness of the map. The result is now an immediate consequence of Proposition \ref{prop A1}.
\end{proof}

\begin{theorem} \label{thm A1} 
Let $\phi\colon Y \to Z$ be a flat morphism of F-finite $\FF_p$-schemes. Suppose $\phi$ has regular fibers, and suppose that $Z$ satisfies $\mathcal P(\lambda)$ 
for some $\lambda \in \R_{\geq 0}$. Then $Y$ satisfies $\mathcal P(\lambda)$.
\end{theorem}
\begin{proof}
By Theorem~\ref{thm semicont} there is a point 
$y \in Y$ such that $\dfpt(Y) = \dfpt(\OO_{Y, y})$. Let $z=\phi(y)$ and  consider the map on stalks $\phi_y\colon \OO_{Z,z} \to \OO_{Y,y}$. 
Since $\phi_y$ is a flat local map of F-finite local rings and its closed fiber is regular by assumption, the F-purity of $\OO_{Z, z}$ implies the F-purity of 
$\OO_{Y, y}$ by Proposition~\ref{prop A1}. Thus we conclude by Corollary~\ref{coroll A1} that $\dfpt(Y) = \dfpt(\OO_{Y,y}) = \dfpt(\OO_{Z,z}) \leq \dfpt(Z) < \lambda$. Hence $Y$ satisfies $\mathcal P(\lambda)$.
\end{proof}

We now focus on Condition (A2).
Let $A \to R$ be a finite type map of Noetherian rings. For $\p \in \Spec(A)$ we let $\kappa(\p) \coloneqq (A/\p)_\p$ and $R(\p) \coloneqq R \otimes_A \kappa(\p)$,  the fiber of $R$ over $\p$. Observe that if we write $R=A[x_1,\ldots,x_d]/I$, then $R(\p) \cong S(\p)/I(\p)$, where $S(\p) = \kappa(\p)[x_1,\ldots,x_d]$ and $I(\p) = I S(\p)$. Similarly, we denote by $R(\overline{\p}) \coloneqq R \otimes_A \overline{\kappa(\p)} \cong S(\overline{\p})/I(\overline{\p})$ the base change of $R(\p)$ to $\overline{\kappa(\p)}$, the algebraic closure of $\kappa(\p)$.

In the context of F-finite and F-pure rings, we recall that for a ring $R=S/I$ we defined $\Theta_e(I) = \max\{n \mid (D^{(n,p^e)}_S(I^{[p^e]}:_S I)) \ne S\}$.

\begin{proposition} \label{prop A2} Let $A$ be an F-finite and F-pure Noetherian integral domain, and let $R=A[x_1,\ldots,x_d]/I$. Assume that $R(\overline{0})$ is F-pure. For any $e \in \NN$ there exists $0 \ne a \in A$, possibly depending on $e$, such that $R(\overline{\p})$ is F-pure and $\Theta_e(I({\overline{\p}})) \leq \Theta_e(I({\overline{0}}))$ for all $\p \in \Spec(A)$ with $a \notin \p$.
\end{proposition}
\begin{proof}
Set $S = A[x_1, \ldots, x_d]$. 
Since both $D^{(n,p^e)}_{S|A}$ and $D^{(n,p^e)}_{S(\overline{\p})|\overline{\kappa (\p)}}$ are free modules with basis given by divided powers differential operators (see Example~\ref{ex: divided powers}), for every $\p \in \Spec (A)$ and $n \in \NN$ we have that 
$
D^{(n,p^e)}_{S(\overline{\p})|\overline{\kappa (\p)}} 
= D^{(n, p^e)}_{S | A} \otimes_S S(\overline{\p})
= D^{(n, p^e)}_{S | A} \otimes_A \overline{\kappa(\p)}.
$
Let $J = I^{[p^e]} :_S I$, and for $n \in \NN$ let
\[
\phi_n \colon D^{(n, p^e)}_{S|A} \otimes_S J \to S
\]
be the evaluation map. 
Note that 
\begin{align*}
D^{(n, p^e)}_{S(\overline{\p})|\overline{\kappa(\p)}} \otimes_{S(\overline{\p})} J(\overline{\p}) & \cong D^{(n, p^e)}_{S | A} \otimes_S S(\overline{\p}) \otimes_{S(\overline{\p})} J(\overline{\p}) \\
& \cong  D^{(n, p^e)}_{S | A} \otimes_S J \otimes_S S(\overline{\p}) \\
& \cong D^{(n, p^e)}_{S | A} \otimes_S J \otimes_A \overline{\kappa(\p)}.
\end{align*}
In particular, $D^{(n,p^e)}_{S(\overline{\p})|\overline{\kappa(\p)}}(J(\overline{\p})) = S(\overline{\p})$ if and only if $\coker(\phi_n) \otimes_A \overline{\kappa(\p)} = 0$. Since $\overline{\kappa(\p)}$ is perfect, by Remark \ref{rem: perfect K} we have that 
\[
\Theta_e(I(\overline{\p})) = \max\{n \mid D^{(n,p^e)}_{S(\overline{\p})|\overline{\kappa(\p)}}(I^{[p^e]}(\overline{\p}):_{S(\overline{\p})} I(\overline{\p}))\}.
\]
Because $J(\overline{\p}) \subseteq I^{[p^e]}(\overline{\p}) :_{S(\overline{\p})} I(\overline{\p})$, from all the above we obtain that
\[
\Theta_e (I(\overline{\p})) +1 \leq 
\min \left \{
n \mid 0 = \coker(\phi_n) \otimes_A \overline{\kappa(\p)} \right \}.
\]
Since the map $A\to \overline{\kappa(0)}$ is flat, so is $S \to S(\overline{0})$. We then have that
$(I^{[p^e]} :_{S} I)  \otimes_A \overline{\kappa(0)} \cong (I^{[p^e]} :_{S} I)  \otimes_S S(\overline{0}) \cong 
I(\overline{0})^{[p^e]} :_{S(\overline{0})} I(\overline{0})$, and hence
\[
\Theta_e (I(\overline{0})) +1 =
\min \left \{
n \mid 0 = \coker \phi_n \otimes_A \overline{\kappa(0)} \right \}.
\]

Set $t = \Theta_e (I(\overline{0})) +1$. By generic freeness, 
there $0 \neq a \in A$ such that $(\coker(\phi_t))_a$ is free.
Since  $\coker(\phi_t) \otimes_A \overline{\kappa(0)} = 0$, 
this implies that $\coker(\phi_t) \otimes_A \overline{\kappa(\p)} = 0$ for any prime $\p \in \Spec (A)$ such that $a \notin \p$. We then have that
\[
\Theta_e (I(\overline{\p})) +1 \leq 
\min \left \{
n \mid 0 = \coker(\phi_n) \otimes_A \overline{\kappa(\p)} \right \}
\leq t =  \Theta_e (I(\overline{0})) +1.
\]
To conclude the proof we need to show that $R(\overline{\p})$ is F-pure for all $\p$ not containing $a$. However, since $R(\overline{0})$ is F-pure by Proposition \ref{prop: global Fedder} we have that $D_{S(\overline{0})}^{(d(p^e-1),p^e)}\left(I^{[p^e]}(\overline{0}):_{S(\overline{0})} I(\overline{0})\right) = S(\overline{0})$. For $\p \in \Spec(A)$ not containing $a$ we have shown that $\Theta_e(I({\overline{\p}}))  \leq \Theta_e(I({\overline{0}})) \leq d(p^e-1)-1$, and it follows that $D_{S(\overline{\p})}^{(d(p^e-1),p^e)}\left(I^{[p^e]}(\overline{\p}):_{S(\overline{\p})} I(\overline{\p})\right) = S(\overline{\p})$. Another application of \ref{prop: global Fedder} gives that $R(\overline{\p})$ is F-pure.
\end{proof}

\begin{proposition} \label{prop1 A2} Let $A$ be an F-finite F-pure integral domain and $R$ be an $A$-algebra of finite type. Assume that $R({\overline{0}})$ is biequidimensional. 
If $R({\overline{0}})$ satisfies $\mathcal P(\lambda)$, then there exists an open subset $V \subseteq \Spec(A)$ such that $R(\overline{\p})$ satisfies $\mathcal P(\lambda)$ for every $\p \in V$.
\end{proposition}
\begin{proof}
Write $R=S/I$ where $S=A[x_1,\ldots,x_d]$, and let $\varepsilon = \frac{\lambda - \dfpt(R({\overline{0}}))}{4}$. Since $\edim((S({\overline{\p}}))_Q) \leq d$ for any $\p \in \Spec(A)$ and $Q \in \Spec(R({\overline{\p}}))$, by Corollary~\ref{uniform convergence fibers} there exists $e \in \NN$ such that
\[
\bigg| \frac{\Theta_{e'}((I({\overline{\p}}))_Q)}{p^{e'}} - \frac{\Theta_e((I({\overline{\p}}))_Q)}{p^e} \bigg| < \varepsilon
\]
for all $e' \geq e$, all $\p \in \Spec(A)$ and all prime ideals $Q \in \Spec(S({\overline{\p}}))$ containing $I({\overline{\p}})$.
By Proposition~\ref{prop A2} there exists an open subset $U \subseteq \Spec(A)$ such that $R(\overline{\p})$ is F-pure and $\Theta_e(I({\overline{\p}})) \leq \Theta_e(I({\overline{0}}))$ for all $\p \in U$. For $e' \geq e$ and $\p \in U$, let $Q_{e'}$ be such that $\Theta_{e'}(I({\overline{\p}})) = \Theta_{e'}((I({\overline{\p}}))_{Q_{e'}})$. Then
\begin{align*}
\Theta_{e'}(I({\overline{\p}})) & = \Theta_{e'}((I_{\overline{\p}})_{Q_{e'}}) \\
& <  p^{e'-e} \Theta_e((I({\overline{\p}}))_{Q_{e'}}) + \varepsilon p^{e'} \\
& \leq p^{e'-e}\Theta_e(I({\overline{\p}})) + \varepsilon p^{e'}  \leq p^{e'-e}\Theta_e(I({\overline{0}})) + \varepsilon p^{e'}.
\end{align*}
Now let $P_e \in \Spec(S)$ be such that $\Theta_e(I({\overline{0}})) = \Theta_e((I({\overline{0}}))_{P_e})$. Then 
\begin{align*}
p^{e'-e}\Theta_e(I({\overline{0}})) & =  p^{e'-e}\Theta_e((I({\overline{0}}))_{P_e}) \\
& < \Theta_{e'}((I({\overline{0}}))_{P_e}) + \varepsilon p^{e'} \leq \Theta_{e'}(I({\overline{0}})) + \varepsilon p^{e'}.
\end{align*}
There exists an open subset $U'$ of $\Spec(A)$ such that $\dim(R(\p)) = \dim(R(0))$ for all $\p \in U'$ \cite[Lemma 37.30.1]{stacks-project}. In particular, this gives that $\height(I({\overline{\p}})) = \height(I({\overline{0}}))$ for all $\p \in U'$. Using Remark \ref{remark DattaSimpson} and \cite[Corollary 2.6.4]{DattaSimpson} there exists an open subset $U'' \subseteq \Spec(A)$ such that $R \otimes_A \overline{\kappa(\p)}$ is biequidimensional for every $\p \in U''$. Let $V=U \cap U' \cap U''$. Since $S(\overline{\p})$ is coequidimensional for any $\p \in \Spec(A)$, 
using Theorem \ref{thm: global (new)} we conclude that 
\begin{align*}
\dfpt(R({\overline{\p}})) & = \lim_{e \to \infty} \frac{\Theta_e(I({\overline{\p}}))}{p^e} - \height(I({\overline{\p}})) \\
& \leq \lim_{e \to \infty} \frac{\Theta_e(I({\overline{0}}))}{p^e} - \height(I({\overline{0}})) + 2\varepsilon' \\
& = \dfpt(R({\overline{0}})) + 2 \varepsilon < \lambda
\end{align*}
for all $\p \in V$. It follows that $R(\overline{\p})$ satisfies $\mathcal P(\lambda)$ for any $\p \in V$.
\end{proof}

\begin{theorem} \label{thm A2} Let $\phi \colon Y \to S$ be a morphism of finite type, where $S$ is F-finite and integral, with generic point $\eta$. Assume that $Y_\eta$ is geometrically F-pure. 
If $Y_\eta$ is equidimensional and geometrically $\mathcal P(\lambda)$, then there exists an open neighborhood $V$ of $\eta$ such that $Y_s$ is geometrically $\mathcal P(\lambda)$ for each $s \in V$.
\end{theorem}
\begin{proof}
By considering an open subset of $S$, we can directly assume that $S = \Spec(A)$ is affine, with $A$ an F-finite F-pure integral domain. By working on finite affine covers of $Y$ we may also reduce to the case where $Y = \Spec(R)$ with $R=A[x_1,\ldots,x_d]/I$. The fact that $Y$ is geometrically $\mathcal P(\lambda)$ gives that $\dfpt(R({\overline{0}})) < \lambda$. Since $Y_\eta$ is equidimensional, so is the affine cover $R(0)$. Since $R(0)$ is of finite type over $\kappa(0)$, using Remark \ref{remark DattaSimpson} we conclude that $R({\overline{0}})$ is biequidimensional. The theorem now follows from Proposition~\ref{prop1 A2}.
\end{proof}

We are now ready to show that a Bertini-type theorem holds for the property $P(\lambda)$.

\begin{corollary} \label{coroll Bertini}
Let $X$ be an F-pure equidimensional quasi-projective subscheme of $\PP^n_\kk$, with $\kk$ algebraically closed of characteristic $p>0$. Then, for any $\varepsilon > 0$, there is an open subset $U \subseteq (\PP^n_\kk)^*$ such that $\dfpt(X \cap H) \leq \dfpt(X) + \varepsilon$
for all hyperplanes $H \in U$.

Moreover, if $\kk$ is uncountable, then 
$\dfpt(X \cap H) \leq \dfpt(X)$ for a very general hyperplane $H \subseteq \PP^n_\kk$.
\end{corollary}
\begin{proof}
For the first assertion, consider the property $\mathcal P(\lambda)$ with $\lambda = \dfpt (X) + \varepsilon$. By Theorem~\ref{thm A1} we have that $\mathcal P(\lambda)$ satisfies (A1). and by Theorem~\ref{thm A2} it also satisfies (A2). This finishes the proof of the first assertion via Theorem~\ref{thm CGM}.
For the second assertion, we apply the first to find open subsets $U_n \subseteq (\PP^n_\kk)^*$ such that 
$\dfpt (X \cap H) \leq \dfpt (X) + 1/n$ for $H \in U_n$. 
Then for any $H \in \cap_{n \in \NN} U_n$ we must have 
$\dfpt (X \cap H) \leq \dfpt (X)$.
\end{proof}

\begin{corollary} \label{coroll subset Bertini}
Let $X \subseteq \PP^n_\kk$ be an equidimensional quasi-projective subscheme of $\PP^n_\kk$, with $\kk$ algebraically closed of characteristic $p>0$. Then, for any $\lambda > 0$, there is an open subset $U \subseteq (\PP^n_\kk)^*$ such that for all $H \in U$
\[
\{x \in X \mid \dfpt ({\mathcal O}_{X, x}) < \lambda\} \cap H \subseteq 
\{x \in X \cap H \mid \dfpt ({\mathcal O}_{X\cap H, x}) < \lambda\}.
\]
\end{corollary}
\begin{proof}
By Remark~\ref{rmk: semicontinuity}, the locus $\{x \in X \mid \dfpt ({\mathcal O}_{X, x}) < \lambda\}$ is open in $X$
and we may apply Corollary~\ref{coroll Bertini} to it. 
\end{proof}

\begin{corollary}\label{coroll Bertini Gor}
Let $X$ be an F-pure normal Gorenstein quasi-projective subscheme of $\PP^n_\kk$, with $\kk$ algebraically closed of characteristic $p>0$. Assume that $X$ is F-finite and F-pure. Then there is an open subset $U \subseteq (\PP^n_\kk)^*$ such that $\dfpt(X \cap H) \leq \dfpt(X)$
for all hyperplanes $H \in U$.
\end{corollary}
\begin{proof}
There is $\varepsilon > 0$ such that 
there is there is no normal Gorenstein local ring $(S, \m)$ of characteristic $p > 0$
such that $\edim (S) \leq n$ and $\dfpt (X) < \dfpt (S) < \dfpt (X) + \varepsilon$ \cite[Theorem~4.7]{SatoACC}. 
By Corollary~\ref{coroll Bertini} there is an open set $U_1 \subseteq (\PP^n_\kk)^*$ such that $\dfpt (X \cap H) \leq \dfpt (X) + \varepsilon$ for all $H \in U_1$. Furthermore, by the Bertini theorem for normality (e.g., by applying \cite[Theorem~5.2]{Flenner} to the closure of $X$)
there is an open set $U_2 \subseteq (\PP^n_\kk)^*$ such that $X \cap H$ is normal whenever $H \in U_2$.
Last, we note that $X \cap H$ is still Gorenstein as long as $H$ does not contain an irreducible components of $X$, giving us condition $U_3$.
It now follows that for any $H \in U_1 \cap U_2 \cap U_3$ we must have 
$\dfpt (X \cap H) \leq \dfpt (X)$.  
\end{proof}

\section{Local properties of the defect of the F-pure threshold} \label{sect: local}

We now list some applications of the methods developed in the previous sections.

\subsection{Behavior for hypersurfaces}

We recall the construction of the Peskine--Szpiro functor on 
 a regular local ring $S$: if $M$ is a finitely generated $S$-module, then 
 $\cF^e_S(M)$ is the additive abelian group of $M$ equipped with the 
 $S$-module structure such that $F_*^e \cF^e_S (M) = M \otimes_S F^e_*(S)$. 

\begin{remark}\label{rmk: Peskine-Szpiro}
By the flatness of Frobenius on $S$, the Peskine--Szpiro functor is exact, so if  $\mathbf{G}_\bullet$ is a finite free resolution of $M$, 
then $\cF^e_S(\mathbf{G}_\bullet)$ is a free resolution of $\cF^e_S(M)$. Explicitly, 
$\cF^e_S(\mathbf{G}_\bullet)$ is obtained from $\mathbf{G}_\bullet$ by
raising to the power $p^e$ all the entries of the matrices representing the maps in $\mathbf{G}_\bullet$.
In particular, if $M = S/J$, then the complex $\cF_S^e(\mathbf{G}_\bullet)$ is a free resolution of $S/J^{[p^e]}$. 
\end{remark}

The following lemma is probably well-known to experts. However, we could not find a precise reference and include a proof for completeness.

\begin{lemma}\label{LemmaColonIdeal}
Let $(S,\m)$ be an F-finite regular local ring, and $I\subsetneq S$ be an ideal such that $R=S/I$ is a Gorenstein ring.
Let $f\in \m$ be a regular element of $R$, and $J=fS+I$.
Then $(J^{[p^e]}:J)=f^{p^e-1} (I^{[p^e]}:I)+J^{[p^e]}$.
\end{lemma}
\begin{proof}
Let $\mathbf{M}_\bullet$ and $\mathbf{N}_\bullet\colon 0\to S \FDer{f} S\to 0$ be minimal free resolutions of $R$ and $S/(f)$ over $S$, respectively.
Let $\mathbf{C}_\bullet=\mathbf{M}_\bullet\otimes_S \mathbf{N}_\bullet$. We note that $\mathbf{C}_\bullet$
is the mapping cone of the map of complexes given by $\mathbf{M}_\bullet \FDer{f} \mathbf{M}_\bullet$. Since $f$ is not a zerodivisor on $R$,  $\mathbf{C}_\bullet$ is a minimal free resolution of $S/J$. 
Let $\cF_S$ denote the Peskine-Szpiro functor on $S$, so that $\cF_S^e(\mathbf{C}_\bullet)=\cF_S^e(\mathbf{M}_\bullet)\otimes_S \cF_S^e(\mathbf{N}_\bullet)$ is a minimal free resolution of
$S/J^{[p^e]}$. Let $\alpha_\bullet \colon \cF_S^e(\mathbf{M}_\bullet)\to \mathbf{M}_\bullet$ and 
 $\beta_\bullet \colon \cF_S^e(\mathbf{N}_\bullet)\to \mathbf{N}_\bullet$ be the maps induced by the natural surjections $S/I^{[p^e]}\to S/I$ and $S/(f^{p^e})\to S/(f)$. We note that $\gamma_\bullet=\alpha_\bullet\otimes_S \beta_\bullet \colon \cF^e_S(\mathbf{C}_\bullet)\to \mathbf{C}_\bullet$ is the map induced by 
the quotient map $S/J^{[p^e]}\to S/J$.
Let $n=\dim(S)$ and $d=\dim(R)$.
Let $h\in S$ be the element such that $\alpha_{n-d}\colon S\to S$ is multiplication by $h$.
We note that $\beta_1\colon S\to S$ is given by the multiplication by $f^{p^e-1}$.
Hence, $\gamma_{n-d+1}=\alpha_{n-d}\otimes \beta_1 \colon S\to S$ is given by multiplication by $f^{p^e-1}h$. We then have that $(I^{[p^e]}:I)=hS+I^{[p^e]}$, and that
 $$
 (J^{[p^e]}:J)=f^{p^e-1}hS+J^{[p^e]}=f^{p^e-1} (I^{[p^e]}:I)+J^{[p^e]},
 $$ 
  because $S/I$ and $S/J$ are Gorenstein \cite[Lemma 1]{VraciuGorTightClosure}.
\end{proof}

The following result is along the lines of the work of Takagi and Watanabe \cite[Proposition~4.3]{TakagiWatanabe}. Compared to their result, in Proposition \ref{prop hyperplane} we remove the assumption of normality, but we require that the ring is Gorenstein, instead of just $\QQ$-Gorenstein.
We point out that it is already known that F-purity deforms for Gorenstein ring \cite[Theorem 3.4]{FedderFputityFsing}. 

\begin{proposition}\label{prop hyperplane}
Let $(R,\m)$ be an F-finite Gorenstein local ring.
If $f\in \m$ is a nonzero divisor such that $R/(f)$ is F-pure, then
$R$ is F-pure, $\fpt(R/(f))\leq \fpt(R)-\ord(f)$, and 
$\dfpt(R)\leq \dfpt(R/(f))$.
\end{proposition}
\begin{proof}
Let $S$ be an F-finite regular local ring mapping onto $R$ \cite{Gabber}, and write $R=S/I$ for some ideal $I \subseteq S$.
Since $R$ is Gorenstein, there exists $g\in S$ such that $g+I^{[p^e]}=I^{[p^e]}:I$. 
By Lemma \ref{LemmaColonIdeal}, we have that $f^{p^e-1}g+J^{[p^e]}=J^{[p^e]}:J$, where $J=I+fS$. 
Then,
$$\Theta_e(J)\geq \Theta_e(I)+\ord(f^{p^e-1})=\Theta_e(I)+(p^e-1)\ord(f),$$
in particular, $R$ is F-pure. Hence,
by Theorem \ref{thm: main formula},
$\fpt(R/(f))\leq \fpt(R)-\ord(f)$.
Since $\ord(f)\geq 1$, we deduce that 
$\dfpt(R)\leq \dfpt(R/(f))$.
\end{proof}

\subsection{Continuity in the $\m$-adic topology}

The Hilbert--Kunz multiplicity and the F-signature are continuous functions with respect to the $\m$-adic topology in certain settings \cite{PolstraSmirnov, PolstraSmirnovE}. We now show that such a property also holds for the defect of the F-pure threshold in Gorenstein rings.

\begin{theorem}\label{thm m-adic}
Let $(R,\m)$ be an F-finite F-pure Gorenstein local ring, and $f_1,\ldots,f_\ell\in \m$ be a regular sequence such that $R/(f_1, \ldots, f_\ell)$ is F-pure.
For every $\varepsilon>0$ there exists $t\in \NN$ such that, for every $h_1, \ldots, h_\ell \in \m^t$, 
\begin{itemize}
\item[(a)] $f_1+h_1,\ldots,f_\ell+h_\ell$ is a regular sequence;
\item[(b)] $R/(f_1+h_1,\ldots,f_\ell+h_\ell)$ is F-pure;
\item[(c)] $|\dfpt(R/(f_1, \ldots, f_\ell))-\dfpt(R/(f_1+h_1,\ldots,f_\ell+h_\ell))|<\varepsilon.$
\end{itemize}
\end{theorem}
\begin{proof}
The first two assertions are already known by previous work
\cite[Lemma~2]{SrinivasTrivedi},
 \cite[Corollary 2.2]{PolstraSmirnovE}
(see also \cite[Corollary 3.10]{PolstraSmirnov}).
It only remains to show the $m$-adic continuity of the defect of the F-pure threshold.
Let $(S,\n)$ be an F-finite regular local ring, and $I\subseteq S$ be an ideal such that $R=S/I$. Let $g\in S$ be such that $g+I^{[p^e]}=I^{[p^e]}:I$.
Given any $\varepsilon > 0$, we now pick $e$ such that $\frac{\dim(S)}{p^e}<\frac{\varepsilon}{2}$.
We take $t \in \NN$ so that the first two assertions hold 
and $\n^{t}\subseteq \n^{[p^e]}$.

By abusing notation we lift 
$f_1, \ldots, f_\ell$ to $S$. We let $J = (I, f_1, \ldots, f_\ell) \subseteq S$, and 
for any tuple $\underline{h} \in \oplus^\ell \m^t$ we 
denote $J_{\underline{h}} \coloneqq I + (f_1+h_1, \ldots,f_\ell+h_\ell)$.
By Lemma \ref{LemmaColonIdeal}, we have that
$$
J_{\ul h}^{[p^e]}:J_{\ul h}=(f_1+h_1)^{p^e-1}\cdots (f_\ell+h_\ell)^{p^e-1} g+J_{\ul h}^{[p^e]}.
$$
Since $\n^{t}\subseteq \n^{[p^e]}$, we deduce that
$(f_1+h_1)^{p^e-1}\cdots (f_\ell+h_\ell)^{p^e-1} gS+ \n^{[p^e]}
=f_1^{p^e-1}\cdots f_\ell^{p^e-1} gS+ \n^{[p^e]}$. It follows that
\begin{align*}
J_{\underline{h}}^{[p^e]}:J_{\underline{h}}+\n^{[p^e]}
=(f_1+h_1)^{p^e-1}\cdots (f_\ell+h_\ell)^{p^e-1} gS+ \n^{[p^e]}
=(J^{[p^e]}:J)+\n^{[p^e]}.
\end{align*}
Therefore $\Theta_e(J)=\Theta_e(J_{\underline{h}})$ and it follows from Corollary~\ref{uniform convergence fibers} that
\[
\left|\dfpt(S/J_{\underline{h}})-\dfpt(S/J) \right|
< 2\frac{\dim(S)}{p^e}
< \varepsilon.
\qedhere
\]
\end{proof}


We now provide an example showing that the assumption that $R$ is Gorenstein in Theorem \ref{thm m-adic} is needed.

\begin{example}[{\cite{FedderFputityFsing,SinghDef}}]\label{ExRef}
Let $\kk$ be a field of characteristic $p>0$, and set $R=S/I$ with $S=\kk\ps{x,y,z,w,t}$ and $I=(xy,xz, z(y+t))$. Let $f=t$, and $h_n=w^n$ for $n\geq 2$. Then $f$ is a regular element in $R$ such that $R/(f)$ is F-pure, but $R/(f+h_n)$ is not F-pure for any $n\geq 2$.
\end{example}

\subsection{F-purity of the associated graded ring}

Let $R$ be a ring, $I \subseteq R$ be an ideal and $M$ an $R$-module. An $I$-filtration of $M$ is a collection of submodules $\mathbb{G} = \{G_n\}_{n = 0}^\infty$ of $M$ such that $M=G_0$,  $G_n \supseteq G_{n+1}$ for all $n$, and $I^n G_m \subseteq G_{n+m}$ for all $n,m \geq 0$. The filtration is called \emph{separated} if $\bigcap_{n \geq 0} G_n = (0)$.
If $\mathbb{G}$ is separated, for $0\ne x \in M$ we let $\ord_{\mathbb G}(x) = \max\{i \mid x \in G_i\}$.

\begin{lemma} \label{lemma filtration}
Let $(R,\m)$ be a local ring and $G,H,M,N$ be finitely generated $R$-modules, equipped respectively with separated $\m$-filtrations $\mathbb{G},\mathbb{H}, \mathbb{M}$, and $\mathbb{N}$. Assume that $G$ is free. If $\varphi \in \Hom_R(G,H)$, $\psi \in \Hom_R(M,N)$ and $\alpha \in \Hom_R(H,N)$ are compatible with the filtrations on the given modules, with $\psi$ surjective, then there exists $\beta \in \Hom_R(G,M)$ compatible with the filtrations on $G$ and $M$ that makes the following diagram commute:
\[
\xymatrix{
G \ar@{-->}[d]_-{\beta} \ar[r]^-\varphi & H \ar[d]^-{\alpha} \\
M \ar@{>>}[r]_-\psi & N
}
\]
In particular, we have an induced commutative square with degree preserving maps:
\[
\xymatrix{
\gr_{\mathbb{G}}(G) \ar[d]_-{\gr(\beta)} \ar[r]^-{\gr(\varphi)} & \gr_{\mathbb{H}}(H) \ar[d]^-{\gr(\alpha)} \\
\gr_{\mathbb{M}}(M) \ar[r]_-{\gr(\psi)} & \gr_{\mathbb{N}}(N)
}
\]
\end{lemma}
\begin{proof}
Let $e_1,\ldots,e_n$ be a basis of $G$ and define $\nu_i = \ord_{\mathbb G} (e_i)$. By the compatibility of maps $\alpha(\varphi(e_i)) \in N_{\nu_i}$ and, since $\psi$ is surjective, there exists $y_i \in M_{\nu_i}$ such that $\psi(y_i) = \alpha(\varphi(e_i))$. Since $G$ is free, we can just define $\beta(e_i) = y_i$ on the basis of $G$. Then the diagram commutes by construction, and moreover
\[
\beta(G_i) = \beta(\sum_{j=1}^n R_{i-\nu_j}e_j) \subseteq \sum_{j=1}^n R_{i-\nu_j}\beta(e_j) \subseteq \sum_{j=1}^n R_{i-\nu_j}N_{\nu_j} \subseteq N_i,
\]
where $R_n = 0$ if $n<0$, $R_0=R$ and $R_n = \m^n$ for $n>0$.
\end{proof}

\begin{proposition} \label{prop initial form Gorenstein}
Let $(S,\m, \kk)$ be an F-finite regular local ring, $I \subseteq S$ be an ideal, and $R=S/I$. Let $T=\gr_\m(S) \cong\kk[x_1,\ldots,x_n]$ and let $J = \IN(I) \subseteq T$, so that $\gr_\m(R) \cong T/J$. Assume that $\gr_\m(R)$ is Gorenstein. Then there exists $f \in S$ such that, for all $e >0,$ one has $I^{[p^e]}:_S I  = (f_e) + I^{[p^e]}$ with $f_e = f^{1+p+\cdots+p^{e-1}}$, and $J^{[p^e]} :_T J = (g_e) + J^{[p^e]}$ with $g_e = \IN(f)^{1+p+\cdots+p^{e-1}}$. Moreover, $g_e \in T$ has degree $(p^e-1)(n+a)$, where $a$ is the $a$-invariant of $\gr_\m(R)$.
\end{proposition}
\begin{proof} A result due to Robbiano \cite{Robbiano} (see also \cite{RS1,RS2}) shows that there is a free resolution $\mathbf{M}_\bullet$ of $S/I$, possibly not minimal, and a separated filtration $\mathbb{M}$ of $\mathbf{M}_\bullet$ such that $\gr_{\mathbb{M}}(\mathbf{M}_\bullet)$ is a minimal free resolution of $T/J$. More specifically, if $M_i = \bigoplus_{j=1}^t Sw_j$ is a free module appearing in $\mathbf{M}_\bullet$, then $\mathbb{M}$ is obtained by appropriately assigning degrees $\nu_1,\ldots,\nu_t$ to the basis elements $w_1,\ldots,w_t$, and then setting $(M_i)_\nu = \bigoplus_{j=1}^t S_{\nu-\nu_j} w_j$, with $S_{<0}=0$, $S_0=S$ and $S_\ell = \m^\ell$ for $\ell>0$.

By applying the Peskine-Szpiro functor $\cF_S^e$ to $M_\bullet$ (see Remark~\ref{rmk: Peskine-Szpiro}), we obtain 
a free resolution $\mathbf{N}_\bullet$ of $S/I^{[p^e]}$. Moreover, for any free module $N_i = \bigoplus_{j=1}^t S w_{e,j}$ appearing in $\mathbf{N}_\bullet$, if we let $\ord(w_{e,j}) = p^e \ord(w_j) = p^e\nu_j$ and set $(N_i)_\nu = \bigoplus_{j=1}^t S_{\nu - p^e\nu_j}w_{e,j}$, then this gives a separated filtration $\mathbb{N}$ of $\mathbf{N}_\bullet$ such that ${\rm gr}_{\mathbb{N}}(\mathbf{N}_\bullet)$ is a minimal graded free resolution of $T/J^{[p^e]}$. This is because ${\rm gr}_{\mathbb{N}}(\mathbf{N}_\bullet)$ coincides with the complex whose maps are given by raising to the power $p^e$ the entries of the matrices representing the maps of ${\rm gr}_{\mathbb M}(\mathbf{M}_\bullet)$, that is, $\cF_T^e({\rm gr}_{\mathbb M}(\mathbf{M}_\bullet))$ (see Remark~\ref{rmk: Peskine-Szpiro}).

The natural map $S/I^{[p^e]} \twoheadrightarrow S/I$ induces a comparison map between the free resolutions which, thanks to a repeated application of Lemma \ref{lemma filtration}, preserves the filtrations:
\[
\xymatrix{
0 \ar[r] & S \ar[r] & \cdots \ar[r] & S \ar[r] & S/I \ar[r] & 0 \\
0 \ar[r] & S \ar[u]^-{\alpha} \ar[r] & \cdots \ar[r] & S \ar@{=}[u] \ar[r] & S/I^{[p^e]} \ar@{->>}[u] \ar[r] & 0
}
\]
We focus on the last map $\alpha\colon S \to S$, and we notice that it is indeed a map between cyclic $S$-modules because $\gr_{\mathbb{M}}(\mathbf{M}_\bullet)$ and ${\rm gr}_{\mathbb N}(\mathbf{N}_\bullet)$ are minimal free resolutions of $T/J$ and of $T/J^{[p^e]}$, respectively, and $T/J$ is assumed to be Gorenstein. As already explained above, if $w_e$ is the chosen basis of the source and $w$ that of the target, then $\ord(w_e) = p^e\ord(w)$. It follows that $\alpha\colon S w_e \to S w$ is the multiplication by an element $f_e \in S_{\nu(p^e-1)}$. If we pass to the associated graded objects, the last map is $\gr(\alpha)\colon \gr_{\mathbb{N}}(S) \to \gr_{\mathbb{M}}(S)$, given by multiplication by $\IN(f_e) \in S_{\nu(p^e-1)}/S_{\nu(p^e-1)+1} \cong T_{\nu(p^e-1)}$. However, the last map in the comparison between the Gorenstein $\kk$-algebras $T/J^{[p^e]} \twoheadrightarrow T/J$ is well-understood: it is multiplication by a homogeneous element $\cdot g_e\colon T(-p^e(n+a)) \to T(-(n+a))$, where $a$ is the $a$-invariant of $T/J$. Our construction yields that $g_e$ coincides with $\IN(f_e) \in T_{\nu(p^e-1)}$, and by comparing degrees we also have that $\nu=n+a$. 
Furthermore, we have that $I^{[p^e]}:_SI = (f_e)+I^{[p^e]}$ and $J^{[p^e]}:_TJ = (g_e) + J^{[p^e]}$ with $g_e=\IN(f_e)$ \cite[Lemma 1]{VraciuGorTightClosure} (see also \cite{DSNBFpurity}). 
Since the comparison map $S/I^{[p^e]} \to S/I$ factors as $S/I^{[p^e]} \to S/I^{[p^{e-1}]} \to \cdots \to S/I^{[p]} \to S/I$, it follows that $f_e = f_1 \cdot f_1^p \cdots f_1^{p^{e-1}}$, as claimed. Our previous considerations give that $g_e = \IN(f_e) = \IN(f_1)^{1+p+\cdots+p^{e-1}} = g_1^{1+p+\cdots +p^{e-1}}$, which has degree $\nu(p^e-1) = (n+a)(p^e-1)$ as desired.
\end{proof}

\begin{remark} We recall that the $a$-invariant of a positively graded $\kk$-algebra $R$, which appears in the last shift of a resolution in the proof of Proposition \ref{prop initial form Gorenstein}, is $a(R) = \max\{j \mid H^{\dim(R)}_\m(R)_j \ne 0\}$, where $\m$ is the ideal generated by the elements of $R$ of positive degree.
\end{remark}

\begin{theorem} \label{thm associated graded}
Let $(R,\m)$ be an F-finite local ring such that $\gr_\m(R)$ is Gorenstein. If $\gr_\m(R)$ is F-pure, then $\fpt(R) =  -a(\gr_\m(R)) = \fpt(\gr_\m(R))$ and, in particular, $R$ is F-pure.

If $\gr_\m(R)$ is Gorenstein and strongly F-regular, then $\s(\gr_\m(R)) \leq \s(R)$ and, in particular, $R$ is strongly F-regular.
\end{theorem}

\begin{proof}
Observe that if $\widehat{R}$ is F-pure then so is $R$, and in this case $\fpt(\widehat{R}) = \fpt(R)$. Similarly, if $\widehat{R}$ is strongly F-regular then so is $R$, and in this case $\s(\widehat{R}) = \s(R)$. Since $\gr_\m(R) \cong \gr_{\widehat{\m}}(\widehat{R})$, we may assume that $R$ is complete. 
We then write $R=S/I$ 
where $S = \kk\ps{x_1,\ldots,x_n}$ and we let $\n = (x_1,\ldots,x_n)$. 
We may assume that $I \subseteq \n^2$.

Let $f$ be as in Proposition \ref{prop initial form Gorenstein}, and let $\nu$ be its $\n$-adic order, so that we can write $f=g+F$ with $F \in \n^{\nu + 1}$ and $g \in \n^\nu \smallsetminus \n^{\nu+1}$. 
Since $\gr_\m(R)$ is F-pure, by Fedder's criterion the support of $g = \IN(f)$ contains a monomial
$y_1^{\beta_1}\cdots y_n^{\beta_n}$ such that $\beta_i<p$ for all $i=1,\ldots,n$. 
It follows from Proposition~\ref{prop initial form Gorenstein} that $R$ is F-pure as well. Note that 
$\fpt(\gr_\m(R)) = -a(\gr_\m(R))$ \cite[Theorem B]{DSNBFpurity}, 
so it is only left to compute the F-pure threshold of $R$.

Let $a=a(\gr_\m(R))$. By Proposition~\ref{prop initial form Gorenstein} 
we obtain that $f_e = f^{1+p+\cdots+p^{e-1}} = g_e+F_e$ where $F_e \in \n^{\nu_e+1}$ and $g_e = g^{1+p+\cdots+p^{e-1}}$ an element of order $\nu_e=(p^e-1)(n+a)$  containing the monomial $\ul{x}^{\ul{\beta}(1+p+\cdots+p^{e-1})} \notin \m^{[p^e]}$ in its support. 
Consider the divided power differential operators 
$\delta_e\coloneqq \partial^{(\ul{\beta}(1+p+\cdots+p^{e-1}))} \in D^{(\nu_e,e)}_S$. Since $\delta_e (F_e) \subseteq \delta_e (\n^{\nu_e+1}) \subseteq \n$, we conclude that $\delta_e(f_e) \equiv \delta_e(g_e) \bmod \n$, and thus $\left(\delta_e(I^{[p^e]}:I) \right)= S$. On the other hand, $D^{(\nu_e-1,e)}_{S}(f_e) \subseteq D^{\nu_e-1}_S(\n^{\nu_e}) \subseteq \n$. It follows that $\Theta_e(I) = \nu_e-1 = (p^e-1)(n+a)-1$, and hence $\fpt(R) = -a$ follows from Theorem \ref{thm: main formula}.

We now proceed to the second assertion. 
Let $T = \kk[y_1,\ldots,y_n]$ and $\eta=(y_1,\ldots,y_n)T$
and write $G \coloneqq \gr_\m(R) = T/J$. By the Fedder-like formula for the F-splitting ideals (see Remark \ref{rem: Fedder}), we have that 
$I_e(R) = \left( (\n^{[p^e]}:_S (I^{[p^e]} :_S I) \right)R
= \left( (\n^{[p^e]}:f_e)\right) R$. Since passing to initial ideals does not change colengths, we may compute that 
\begin{align*}
\ell_R\left(R/I_e(R)\right)
= \ell_{S}\left(S/(\n^{[p^e]}:f_e)\right)
= \ell_{\gr_\m(S)}\left(\gr_\m(S)/(\IN(\n^{[p^e]}:f_e))\right).
\end{align*}
After noting that $\eta^{[p^e]} = \IN(\n^{[p^e]})$, we obtain a similar expression
\[
\ell_G (G/I_e(G)) = 
\ell_T\left(T/(\eta^{[p^e]}:g_e)\right)
=\ell_{\gr_\m(S)}\left(\gr_\m(S)/(\IN(\n^{[p^e]}):\IN(f_e))\right).
\]
Since $\IN(\n^{[p^e]}:f_e) \subseteq \IN(\n^{[p^e]}):\IN(f_e)$, 
we immediately obtain that $\ell_G (G/I_e(G)) \leq \ell_R\left(R/I_e(R)\right)$.
Because  $\dim (R) = \dim (G)$, the desired bound is now immediate from the definition of F-signature:
\begin{align*}
\s(G)
= \lim_{e\to \infty} \frac{\ell_{\gr_\m(S)}\left(\gr_\m(S)/I_e(G)\right)}{p^{e(\dim (G))}} 
\leq \lim_{e\to \infty} \frac{\ell_{R}\left(R/I_e(R)\right)}{p^{e(\dim (G))}}  = \s(R).
\end{align*}
Finally, since positivity of the F-signature characterizes strong F-regularity \cite{AL}, the above inequality concludes the proof.
\end{proof}

The following example, suggested by the referee, shows one use of Theorem \ref{thm associated graded}. 
\begin{example}
Let $\kk$ be a field of characteristic $p>0$, $T=\kk[x_1,\ldots, x_d] $ with the standard grading, and $\m=(x_1,\ldots, x_d)$.
Let $f\in T$ be an homogeneous element such that $G= T/(f)$ is F-pure. Let $\epsilon\in \m^{\deg(f)+1}.$  If $R=T_\m/(f+\epsilon)T_\m$, then $G\cong \gr_\m(R)$ and $\dfpt(G)=\dfpt(R)$
by Theorem \ref{thm associated graded}. Furthermore, if $G$ is strongly F-regular, then $s(G)\leq s(R)$ by Theorem \ref{thm associated graded}.
\end{example}

We note that the assumption that $\gr_\m(R)$ is Gorenstein is needed in Theorem \ref{thm associated graded} in order to conclude that $R$ is F-pure when $\gr_\m(R)$ is F-pure. This is not too surprising since passing from $\gr_\m(R)$ to $R$ can be seen as a flat deformation $\kk[T^{-1}] \to R[\m T,T^{-1}]$ from the special fiber to a generic fiber, and it is known that F-purity deforms for Gorenstein rings \cite{FedderFputityFsing} (or, more generally, for $\QQ$-Gorenstein rings \cite{PolstraSimpson,SchwedeFAdj}), but not in general.

\begin{example}[{\cite{FedderFputityFsing,SinghDef}}]\label{ExRef}
Let $\kk$ be a field of characteristic $p>0$, $T=\kk\ps{x,y,z,w}$ and $I=(xy,xz,z(y-w^n))$ with $n\geq 2$. Let $R=T/I$ and $\m=(x,y,z,w)R$. Then $R$ is not F-pure though ${\rm gr}_\m(R) \cong \kk[x,y,z,w]/(xy,xz,yz)$ is F-pure.
\end{example}

We apply Theorem~\ref{thm associated graded} to control the defect of the F-pure threshold along the blowup at the maximal ideal in the case in which the associated graded ring is F-pure and Gorenstein.

\begin{proposition} \label{prop: blowup}
Let $(R,\m)$ be an F-finite local ring,  let $\pi\colon X \to \Spec (R)$ be the blowup of $R$ at $\m$, and $S$ be any local ring of $X$. 
If $\gr_\m(R)$ is Gorenstein and F-pure, then 
$\dfpt(S) \leq \dfpt(R)$. 
\end{proposition}
\begin{proof}
By definition, 
$X$ is the Proj of the extended Rees algebra  $\cR = R[\m T, T^{-1}]$ and $S$ is its homogeneous localization at a homogeneous prime ideal $P\in \Proj(\cR)$. By Theorem~\ref{thm semicont} we may assume that $S$ corresponds to a closed point and, since $\pi$ is an isomorphism over the punctured spectrum, 
we may assume that $P$ contracts to $\m$, i.e., $T^{-1} \in P$. 
Note that $\cR /T^{-1}\cR \cong \gr_\m(R)$, so $\cR_P /T^{-1}\cR_P = (\gr_\m(R))_P$ is Gorenstein and F-pure. 
Since $T^{-1}$ is a regular element, 
$\dfpt (\cR_P) \leq \dfpt((\gr_\m(R))_{P})$ by Proposition~\ref{prop hyperplane}. 
Furthermore, by the localization inequality and Theorem~\ref{thm associated graded},
\[
\dfpt((\gr_\m(R))_{P}) \leq \dfpt(\gr_{\m} (R)) = \dfpt (R).
\]
It is left to compare the homogeneous localization $S = \cR_{(P)}$ with $\cR_P$.

Because $P$ is a point of $\Proj(\cR)$, it does not contain  $x = fT$ for some $f \in \m$. Hence,  $\cR_P$ is the localization of $A = \cR[x^{-1}]$ at the image of $P$. If we let $B = \{gx^{-n} \mid n \in \NN,$ $g \notin P$ homogeneous with $\deg(g)=n\}$, then $B$ is the subring of $A$ consisting of homogeneous elements of degree zero. Note that $A = B [x, x^{-1}]$; in fact, for any homogeneous $g \in \cR$ of degree $n$ one can write $g=(gx^{-n})x^n$. We akso have that $\cR_{(P)}$ is the localization of $B$ at the prime ideal $Q  = PA \cap B$.  Note that $QA = Q[x, x^{-1}]$ is still a prime ideal and $A_{QA} = B[x, x^{-1}]_{Q[x, x^{-1}]} \cong B_Q (y) \cong \cR_{(P)}(y)$, where $y$ is a variable over $B_Q$ and $B_Q(y) = (B_Q[y])_{QB_Q[y]}$. Similarly for $\cR_{(P)}(y)$.
Since the defect of the F-pure threshold can only decrease when localizing we get
\[
\dfpt (\cR_{(P)}) = \dfpt (\cR_{(P)}(y))
= \dfpt(A_{QA}) \leq \dfpt(A_{PA}) = \dfpt(\cR_P),
\]
and the assertion follows.
\end{proof}

\subsection{Behavior under flat extensions}
\label{sub: flat}

The intuition about how invariants typically behave under faithfully flat maps suggests that, given a flat local map of rings $(A,\m_A) \to (R,\m_R)$, we should expect the inequalities
\begin{equation}
\label{ineqFlat}
\dfpt (A) \leq \dfpt (R) \leq \dfpt (A) + \dfpt (R/\m_A R).
\end{equation}
However, some assumptions are needed in order for this to hold. Indeed, as remarked by Hashimoto \cite[Remark~2.17]{Hashimoto}, Singh's counterexample to deformation of strong F-regularity gives a faithfully flat map $k[T]_{(T)} \to R$ such that $R$ is not F-pure but the closed fiber is strongly F-regular \cite[Proposition~4.3 and Proposition~6.2]{singh99}.
Thus, it is only natural to assume that $R$ is F-pure. Note that  $A$ is also F-pure in this case \cite[Proposition~5.13]{HRFpurity}.

\begin{remark}
It is known that F-purity descends under arbitrary pure maps $A \to R$ \cite[Proposition~5.13]{HRFpurity},
so it is natural to ask if this result can be generalized to F-pure thresholds. 
However, the naive inequality $\dfpt (R) \geq \dfpt (A)$ does not hold: 
if we take $R = \kk\ps{x,y}$ and $A = \kk\ps{x^2, xy, y^2}$ 
then $\dfpt (R) = 0$ and 
$\dfpt (A) = 1$  \cite[Example~2.4]{TakagiWatanabe}. 
\end{remark}
If we add conditions on the closed fiber of the map, we do get some positive answers as in Corollary \ref{coroll A1}.
Furthermore, thanks to Corollary \ref{coroll A1} we can prove one Inequality (\ref{ineqFlat}) only assuming that the closed fiber is reduced.

\begin{proposition} \label{prop: ineq flat 1}
Let $(A,\m_A) \to (R,\m_R)$ be a faithfully flat map of F-finite and F-pure local rings. 
If $R/\m_A R$ is reduced,  then  
$\dfpt(R) \geq  \dfpt(A)$.
\end{proposition}
\begin{proof}
Let $P$ be a minimal prime of $\m_A R$ such that $\dim(R_P) = \dim(A)$. 
Then $\dfpt (R_P) \leq \dfpt (R)$ by Theorem \ref{thm semicont}, so it suffices to assume that $P = \m_R$. 
In this case, because $R/\m_A R$ is Artinian and reduced, it must be a field, 
so $\dfpt (R) = \dfpt (A)$ by Corollary~\ref{coroll A1}.
\end{proof}

Building on Ma's argument \cite[Proposition~5.4]{MaLC}, we settle a particular case of the second inequality in (\ref{ineqFlat}).

\begin{theorem} \label{gorenstein flat}
Let $(A,\m_A) \to (R,\m_R)$ be a faithfully flat map of F-finite and F-pure local rings with $R/\m_A R$ Gorenstein and F-pure. Then  
$\dfpt(R) \leq  \dfpt(A)+\dfpt(R/\m_A R)$.
\end{theorem}
\begin{proof}
Set $S \coloneqq R/\m_A R$. 
Take elements $f\in \m^{b_e(A)}\setminus I_e(A)$ and $g\in \m_R^{b_e(S)}$ such that its class belongs to $\m_R^{b_e(S)}S\smallsetminus I_e(S)$. 
It suffices to show that 
$fg \notin I_e(R)$.
Namely, because
\[
fg \in \m_A^{b_e(A)} \m_R^{b_e(S)}
\subseteq \m_R^{b_e (A) + b_e(S)}, 
\]
we would have $b_e(R)\geq b_e(A) + b_e(S)$, so 
by dividing by $p^e$ and taking limits as $e \to \infty$, 
it follows that $\fpt(R)\geq \fpt(A) +\fpt(R/\m_A R)$.
The assertion then holds by the dimension formula. 

Since $A$ is excellent and reduced, it is approximately Gorenstein  \cite{Hochster}, so 
let $	\{J_k\}_{k = 1}^\infty$ be an approximating sequence of ideals. Take $x_1,\ldots,x_t 	\in R$ be elements that give a system of parameters in $S$. 
We have that  $\a_k=J_k R+(x^k_1,\ldots, x^k_t)$ is an approximating sequence of ideals in $R$ \cite[proof of Theorem 7.24]{HoHu2}.
Let $v_k\in A$ be socle representatives of $A/J_k$ and $w_k\in R$ be a socle representative of  $R/(\m_A,x^k_1,\ldots, x^k_t)R$.
Now, recall that the injective hull of the residue field can be written as the direct limit $E_A = \varinjlim A/J_k$. Thus, via the injective hull criterion for splitting of maps (\emph{e.g.,} \cite[Lemma~1.6]{TakagiWatanabe}), we may  
restate the needed claim
using the approximating sequences: 
given that 
$fv_k^{p^e} \notin J_k^{[p^e]}, gw_k \notin (\m_A, x_1^{kp^e}, \ldots, x_t^{kp^e})$ for all $k \in \NN$ we want to show that 
$fgv^{p^e}_k w^{p^e}_k\not\in \a^{[p^e]}_k$ for every $k\in\NN$.

By contradiction, suppose that $fgv^{p^e}_k w^{p^e}_k \in \a^{[p^e]}_k$. Let $B=R/(x^{k p^e}_1,\ldots, x^{k p^e}_t)$. We have that 
$fgv^{p^e}_k w^{p^e}_k\in J^{[p^e]}_k B$. Then,
$gw^{p^e}_k\in (J^{[p^e]}_k B: fv^{p^e}_k)=\frac{(J^{[p^e]}_k : fv^{p^e}_k)}{(x^{k p^e}_1,\ldots, x^{k p^e}_t)}$,
where the equality follows from the fact that $B$ is faithfully flat over $A$.
It follows that $gw^{p^e}_k\in (J^{[p^e]}_k : fv^{p^e}_k)+(x^{kp^e}_1,\ldots, x^{kp^e}_t)$ and, by our choice of $f$, we have that $(J^{[p^e]}_k:fv^{p^e}_k)\subseteq \m_A$. Finally, we have that
$\overline{g}\overline{w}^{p^e}_k\in (x^{kp^e}_1,\ldots, x^{kp^e}_t)S$.
This contradicts the choice of $g$.
\end{proof}

We do not know whether the inequality $\dfpt (R) \leq \dfpt (A) + \dfpt (R/\m_A R)$ holds in Theorem \ref{gorenstein flat} without the assumption that $R/\m_A R$ is Gorenstein. 

We conclude the section with an example which shows that we cannot expect the equality $\dfpt(R) = \dfpt(A) + \dfpt(R/\m_A R)$ to hold, even in the assumptions of Theorem \ref{gorenstein flat}. In particular, the closed fiber being regular is crucial in Proposition \ref{prop A1}.

\begin{example} Let $\kk$ be an F-finite field. Let $A=\kk\ps{t}$,  $R=\kk\ps{x,y}$, and consider the map $f\colon A \to R$ such that $t \mapsto xy$. Since $xy$ is a regular element in $R$ we have that $f$ is faithfully flat, and the closed fiber $R/\m_A R \cong R/(xy)$ is F-pure and a hypersurface, hence Gorenstein. However, $\dfpt(A)=\dfpt(R) = 0$ and $\dfpt(R/(xy)) = 1$.
\end{example}

\subsection{Singularities in characteristic zero}\label{ssect: modp}
There is a (in part conjectural) correspondence between 
log-canonicity of a singularity of characteristic $0$ and 
F-purity of its \emph{reductions mod} $p > 0$. 
We now give a brief review of the construction of reductions mod $p > 0$. 
We refer to the unpublished manuscript of Hochster and Huneke on tight closure in equal characteristic zero \cite[Section 2.1]{HHChar0} for more background information.

We assume that $R$ is of finite type over $\CC$ and $\p$ is a prime ideal of $R$.
We  work with \emph{models} $(\cR, \p_{\cR})$ of $(R, \p)$:
\begin{enumerate}
\item $A$ is a $\ZZ$-subalgebra of $\CC$,
\item $\cR$ is finitely generated $A$-algebra,
\item $R = \cR \otimes_A \CC$, 
\item $\p_{\cR}$ is a prime ideal and $\p_{\cR} R = \p$. 
\end{enumerate}
For any maximal ideal $\a$ of $A$,  
the ring $R (\a) \coloneqq \cR \otimes_A A/\a$
has positive characteristic and is called a \emph{reduction modulo $p$}
of $R$. 
Many properties of $R$ can be preserved in $\cR$ and $R(\a)$ by appropriately choosing $A$ and $\a$. 
In particular, we have  that 
$\p R(a)$ is a prime ideal for $a$ in a Zariski dense open subset 
of $\Spec A$ \cite[Theorem~2.3.6]{HHChar0}.

\begin{definition}
We say that a pair $(R_\p, \p^t)$ is of dense F-pure type\footnote{One can also define dense F-pure type using ultra-Frobenii, see \cite{Yamaguchi}.} if 
there exists a dense subset $U$ of $\Spec (A)$ consisting of maximal ideals 
such that for any $\a \in U$ the pair 
$(R(\a)_{\p_\cR R(\a)}, \p_{\cR} R(\a)^t)$ is F-pure in the sense of Definition~\ref{DefFPure}. 
\end{definition}

A $\QQ$-Gorenstein normal ring 
of dense F-pure type has log canonical singularities \cite{HaraWatanabe}
and the converse is conjectured to be true. 
It is known \cite{TakagiAdjoint} that the converse follows from the \emph{weak ordinarity conjecture}  \cite{MustataSrinivas}.
The conjecture would imply that the log canonical threshold $\lct (\p R_\p)$ 
is equal to the dense F-pure threshold and thus can by approached
using F-pure thresholds in the reductions mod $p$. 
The two notions are known to coincide if $R_\p$ has a log terminal singularity  \cite[Proposition~3.2]{TakagiWatanabe},
due to the correspondence between multiplier and test ideals.
The results of this paper imply several properties of log canonical thresholds.

\begin{theorem}
Let $R$ be a $\QQ$-Gorenstein log terminal ring of 
finite type over $\CC$ and $S$ be a finite type $R$-algebra 
which is also $\QQ$-Gorenstein and log terminal.
Then:
\begin{enumerate}
\item for prime ideals $\p \subseteq \m$ of $R$ 
one has the inequality
\[
\lct (\p R_\p) + \dim (R_\m/\p R_\m) \geq \lct (\m R_\m);
\]
\item 
if $\m \subseteq R$ and $\n \subseteq S$ are prime ideals such that 
$R_\m \to S_\n$ is a flat local map and $S_\n/\m S_\n$ is  regular,
then $\lct (\m R_\m) + \dim (S_\n/\m S_\n) = \lct (\n S_\n)$;
\item 
if $\m \subseteq R$ and $\n \subset S$ are prime ideals such that 
$R_\m \to S_\n$ is a flat local map and $S_\n/\m S_\n$ is reduced, 
then $\lct (\m R_\m) + \dim (S_\n/\m S_\n) \geq \lct (\n S_\n)$.
\end{enumerate}
\end{theorem}
\begin{proof}
For the first property, we choose a model for both $\p$ and $\m$ (see \cite[(2.1.10) and (2.1.14)]{HHChar0})
and use Theorem~\ref{thm semicont} to get an inequality 
$\dfpt (R(\a)_\p) \leq \dfpt (R(\a)_\m)$ for all maximal ideals $\a \subseteq A$ for which  $R(\a)_\m$ is F-pure. 

Now we prove the remaining assertions.
As explained in by Hochster and Huneke \cite[(2.1.18)]{HHChar0}, we can find compatible models  $(\cR, \m_{\cR})$, $(\cS, \n_{\cS})$ for 
$R$ and $S$:
there is a finitely generated $\ZZ$-algebra $A$ and 
a finite type homomorphism $f \colon \cR \to \cS$ of finitely generated 
$A$-algebras, such that $R \to S$ is the localization of $f$ at the generic point of $A$.
By the assumption $\m = R \cap \n$, so $S_\n$ is a flat $R$-module. 
Since $\n_{\cS} \cap A = 0$ by the construction, 
we may restate this as flatness of $\cS_{\n_\cS} = S_\n$ over $\cR$.  
By the openness of the $\cR$-flat locus  \cite[00RC]{stacks-project}, 
we can find an element $g \notin \n$ 
so that $\cS_g$ is a flat $\cR$-module. By base change, we then get that 
$S(\a)_g$ is a flat $R(\a)$-module, and it follows that $R(\a)_\m \to S(\a)_\n$ is a flat map.

It remains to show that we may preserve the special properties of the fibers. 
First, if $S_\n/\m S_\n$ is regular, then, by the openness of the regular locus, 
there is $g \notin \n$ such that $S_g/\m S_g$ is regular. 
Since $\cS_g$ is a finitely generated $A$-algebra,  
there is an open set of maximal ideals $U$ of $A$ such that 
$S_g(\a)/\m S_g(\a)$ is geometrically regular over the residue field of $\m R(\a)$ for all $\a \in U$ \cite[Theorem 2.3.6]{HHChar0}.
Thus, we may apply Corollary~\ref{coroll A1} in these fibers,
and the second assertion follows. 
The proof of the third assertion is similar. We utilize reduction to prime characteristic \cite[Theorem 2.3.16(c)]{HHChar0}
to verify that there is an open set of maximal ideals $V$ of $A$ 
such that $\m S_g(\a)$ is radical for all $\a \in V$. 
\end{proof}

Another potential approach for converting our results to characteristic $0$ 
is by developing an analogue of F-pure thresholds using big Cohen--Macaulay algebras, as an extension of the theory of BCM-thresholds developed by Schwede and Rodr{\' i}guez Villalobos \cite{KarlSandra}. Such theory may provide a natural characteristic-free invariant for which our results may remain true. 

\section{Examples, questions and further remarks}\label{sect: last}

\subsection{Examples}
We now give an example of a quotient by a monomial ideal where $\p \to \dfpt(\p)$ and $\p \to \mfpt(\p)$ are maximized at different maximal ideals.

\begin{example} Let $S=\kk[x_1,\ldots,x_5,y]$ and $I = (y) \cap (x_1,x_2,x_3) \cap (x_3,x_4,x_5)$. Let $W=S \smallsetminus \left(\p_1 \cup \p_2\right)$, where $\p_1 = (x_1,x_2,x_3,x_4,y)$ and $\p_2 = (x_1,x_2,x_3,x_4,x_5)$. Let $R=(S/I)_W$. Note that 
\[
R_{\p_1} \cong \left(\frac{S}{(y) \cap (x_1,x_2,x_3)}\right)_{\p_1} \text{ and } R_{\p_2} \cong \left(\frac{S}{(x_1,x_2,x_3) \cap (x_3,x_4,x_5)}\right)_{\p_2},
\]
so that $\fpt(R_{\p_1}) = 1$ and $\fpt(R_{\p_2}) = 0$. It follows that
\[
\mfpt(\p_1) = 5-1 = 4 \ \ \text{ and } \ \ \dfpt(R_{\p_1}) = 4-1 = 3,
\]
while
\[
\mfpt(\p_2) = 5-0 = 5 \ \ \text{ and } \ \ \dfpt(R_{\p_2}) = 2-0 = 2.
\]
\end{example}

Second, we observe that the function 
$\Spec(R)  \ni \p \mapsto \dim(R_\p) - c^{\p R_\p}(\p R_\p)$ may not be upper semi-continuous. 

\begin{example}
Let $\kk$ be a field of characteristic $p > 2$, $S = \kk\ps{x,y,z,w}$, $f = x^2 - w^2(y^2 + z^2)$, and $R = S/(f)$. 
For 
$f^{p-1}$ contains the monomial $x^{p-1} w^{p-1} y^{p-1}$ in its support,
$R$ is F-pure by Fedder's criterion. 
Furthermore, $y,z,w$ is a system of parameters and $\m$ is its tight closure 
since $x x^{p^e} \in (w^{p^e})$. On the other hand, 
 the localization at $\q = (x,y,z)$ is F-rational as it is an $(A_1)$-singularity.

Recall that $c^\m (\m) \leq \dim (R)$ \cite[Proposition~2.2 and Corollary~3.2]{HMTW}, and that equality holds if and only if there is a minimal reduction $J$ of $\m$ such that $J^* = \m$. Therefore in this case $c^\m (\m) = 3$ while $c^{\q R_\q} (\q R_\q) < 2$, 
showing that the function $\p \in \Spec(R) \mapsto \dim(R_\p) - c^{\p R_\p} (\p R_\p)$ is not upper semi-continuous. 

On the other hand, we have that 
$\dfpt (R_\q) = 1$ \cite[Example~2.5]{TakagiWatanabe}. A direct calculation shows that  
\[f^{p^e - 1} \cong  x^{p^e - 1}w^{p^e-1} (y^2, z^2)^{\frac{p^e-1}2} \pmod {\m^{[p^e]}}, 
\]
and since 
\[
\binom{p^{e}-1}{\frac{p^e-1}{2}} \cong \binom {p-1}{\frac{p-1}{2}}^{e} \neq 0 \pmod p
\]
by Lucas' Theorem, we conclude that $\Theta_e(fS) = \max\{n \mid f^{p^e-1} \in \m^n+\m^{[p^e]}\} = 3(p^e-1)$. It follows that $\dfpt (R_\m) = 2$.  
\end{example}

\subsection{Behavior in families}\label{sub: thanks Karen}
Given a flat morphism $B \to R$ of F-finite rings of prime characteristic $p>0$
we may define a function on $\Spec(B)$ 
by sending $\p \mapsto \dfpt (R \otimes_B k(\p))$ 
provided that the assumptions of Theorem~\ref{thm: global (new)} are satisfied.
It is natural to ask whether this function is also semi-continuous. 

A particular case is when there is an ideal $I \subseteq R$ such that $R/I \cong B$.  Then the image of $I$ in each $R(\p) \coloneqq R \otimes_B k(\p)$ is a maximal ideal, 
so $\dfpt (R(\p))$ can be defined as the defect of the F-pure threshold at such maximal ideal. In this way, we may treat it as a function $\p \in \Spec(B) \mapsto \dfpt(R(\p))$.
Note that this function can be recovered from the 
F-pure threshold of pairs $\Spec(B) \ni \p \mapsto \fpt (IR(\p), R(\p))$. 
In an unpublished manuscript \cite{Liu}, Yuchen Liu gives several results 
on lower semi-continuity of F-pure threshold of triples
that should provide  some results on upper semi-continuity of the defect of F-pure threshold in families of local rings. 

A particular type of families 
is the moduli of homogeneous hypersurfaces of fixed degree. 
A recent work of Smith and Vraciu \cite{SmithVraciu}
studies the stratification of this moduli by the F-pure threshold of the principal ideal defining the hypersurface. The also find a formula for the generic value. On the contrary, for our invariant the stratification is trivial. Namely, the F-pure threshold of a Gorenstein standard graded algebra over a field $\kk$ is equal to the opposite of its $a$-invariant \cite[Theorem B (3)]{DSNBFpurity}. In particular, the value $\dfpt (\kk[x_1, \ldots, x_d]/(f))$ only depends on whether $f$ defines an F-pure ring or not.

\subsection{Discreteness for non-Gorenstein rings}
Due to a result of Sato \cite{SatoACC}, we were able to obtain stronger results for Gorenstein rings in Theorem~\ref{thm constructible} and Corollary~\ref{coroll Bertini Gor}. A natural question is to ask whether some form of the Gorenstein property is truly necessary. Note that our results do not require the full power of Sato's ACC condition, since the local rings we deal with come from ``one source''.  

In the same spirit, we expect that Theorem~\ref{thm m-adic} 
can be improved, either by removing the Gorenstein assumption or by getting a better control on the F-pure threshold of the perturbations. 
For example, if we were able to control normality in perturbations or 
reprove \cite{SatoACC} without this assumption, then 
we would immediately get that $\dfpt (R)/(f_1, \ldots, f_\ell)
\geq \dfpt (R/(f_1+h_1, \ldots, f_\ell + h_\ell))$ for $h_i$ chosen sufficiently deep inside the maximal ideal of $R$.

\subsection{Finiteness for the defect of the F-pure threshold}

 If $R$ is an algebra essentially of finite type over the complex numbers with at worst KLT singularities, then the anticanonical cover is finitely generated \cite{BCHM}. Since strongly F-regular singularities are the prime charactersitic counterpart to KLT singularities \cite{HaraWatanabe,Takagi04}, it is expected that strongly F-regular rings have a finitely generated anticanonical cover as well.
For these reasons, one would expect that the claims of Theorem \ref{thm constructible} hold for strongly F-regular rings. These considerations motivate the following questions:
\begin{question}
Let $R$ be an F-finite strongly F-regular ring. 
\begin{enumerate}
\item Is the set $\{\fpt(R_\q) \mid \q \in \Spec (R)\}$ finite? 
\item Are the functions 
\[
\q \mapsto \dfpt (R_\q) \ \text { and } \ \q \mapsto \mfpt(R_\q) \coloneqq \edim (R_\q) - \fpt (R_\q),
\]
 strongly upper semi-continuous? 
 \end{enumerate} 
\end{question}

\bibliographystyle{alpha}
\bibliography{References}
\end{document}